\documentclass{dpreprint}
\begin{document}

\dtitle{Mixing on Stochastic Staircase Transformations}

\dauthor[Darren Creutz]{Darren Creutz}{creutz@usna.edu}{US Naval Academy}

\datewritten{7 April 2020}
\dsubjclass{Primary 37A25; Secondary 28D05}

\dabstract{%
We prove mixing on a general class of rank-one transformations containing all known examples of rank-one mixing, including staircase transformations and Ornstein's constructions, and a variety of new constructions.
}

\makepreprint

{\allowdisplaybreaks
\section{Introduction}

%\subsection{Rank-One Transformations}
The study of rank-one transformations, particularly mixing rank-one transformations, is an active area of research \cite{DS04}, \cite{AP06}, \cite{ryz1}, \cite{ageev}, \cite{ryz2}, \cite{DJ08}, \cite{CS10}.  These transformations are in some sense the simplest constructive class complex enough to include examples of various mixing properties.  Originally constructed by Ornstein \cite{Or72} to provide examples of mixing transformations with no roots, rank-one transformations are now a source of examples and counterexamples for many mixing-related properties.

%\subsection{Mixing on Rank-One Transformations}
Recently, in \cite{CS04} and \cite{CS10}, an approach to proving mixing on rank-one transformations was developed involving a sort of bootstrapping process.  First one shows that a certain sequence (the spacer sequence) is ergodic for all ergodic transformations and deduces from this that the transformation is weakly mixing.  In turn this is used to show that a certain family of sequences (the partial sums) are ergodic with respect to the transformation leading to a large set on which mixing occurs.  This process continues in a similar back and forth, culminating in mixing.

%\subsection{Randomly Generated Rank-One Transformations and Mixing}
Ornstein's original construction of mixing transformations used uniformly distributed random variables so that almost surely the resulting transformations are mixing.  Adams \cite{Ad98} showed that a deterministic class, a subset of the staircase transformations, are also mixing and the author and Silva \cite{CS10} extended this result to all staircases.  Ornstein's class has been extended in various ways, for instance in \cite{El00} and \cite{AP06} to larger classes and in \cite{Da06} to various abelian group actions.

%\subsection{Stochastic Staircase Transformations}
We introduce the class of stochastic staircase transformations, a broad generalization of Ornstein's random constructions which includes the deterministic staircase transformations and all other known examples of rank-one mixing transformations, and our main result is the following:
\begin{theorem*}[Theorems \ref{T:rsmix} and \ref{T:rsmix2}]
Almost every stochastic staircase transformation is mixing.
\end{theorem*}
This provides a large class of mixing rank-one transformations including, in addition to those just mentioned, the random staircase transformations, the random analogue of the staircase construction with positive density spacer sequences.

%\subsection{Mixing Properties and Ergodicity of Stochastic Sequences}
The main new ingredient in our work here is a new result about a strong type of uniform ergodicity for a general type of randomly generated sequences (Theorem \ref{T:P}).  Lema\'{n}czyk, Lesigne, Parreau, Voln\'{y} and Wierdl \cite{leman} proved that for a very large class of sequences--sequences determined by the behavior of measurable functions under measure-preserving transformations, a class which includes the sequences we study here--the mean ergodic theorem holds along such sequences for every ergodic transformation.  Using their result as the replacement for the mean ergodic theorem, and our new result on the uniform ergodicity behavior of stochastic sequences for the later steps, we follow the approach of \cite{CS10} to show that stochastic staircase transformations are mixing.

This uniform ergodicity of sequences is actually a spectral result about sequences and rank-one transformations.  The spectral behavior of rank-one transformations has been the subject of much study, notably the result of Klemes and Reinhold \cite{klemes} that a class of rank-one transformations have singular spectra, as well as the results in \cite{bourgain93}, \cite{klemes96}, \cite{abda07} and \cite{aist12}.  Mixing rank-one transformations are known to be mixing of all orders \cite{Ka84}, \cite{Ry93} and so results on the singularity of the spectrum of rank-one transformations connect to the result of Host \cite{host} that mixing implies mixing of all orders when the spectrum is singular.  As rank-one transformations are known to have simple spectra and the necessity of mixing for Lebesgue spectra is obvious, the class of mixing rank-one transformations is of interest as a potential source of examples of solutions to the well-known Banach problem on the existence of transformations with simple Lebesgue spectrum.

%\subsection {Acknowledgments.}  
\textbf{Acknowledgements.}  The author would like to thank C.~Dodd and B.~Robinson for their contributions during early investigation of this problem and also to thank C.~Silva.  The author would also like to thank the referees on initial drafts of the paper for many helpful suggestions, especially the suggestion of exact references and regarding the organization of the paper and for various proof simplifications.
This paper is based in part on 
research conducted during the 2004 SMALL Undergraduate Summer
Research Project at Williams College.  Support for the project was provided by a National Science 
Foundation REU Grant and the Bronfman Science Center of Williams 
College.

\section{Preliminaries}
\subsection{Dynamical Systems}
A probability space $(X,\Sigma,\mu)$ together with a measure-preserving,
invertible transformation $T: X \mapsto X$ form a \textbf{dynamical 
system} $(X,\Sigma,\mu,T)$.  The term \textbf{transformation} will refer 
exclusively to such $T$.
A transformation $T$ is \textbf{ergodic} if for any 
$A\in\Sigma$, if $T(A)=A$ then
$\mu(A)\mu(A^{C})=0$. The \textbf{ mean (von Neumann) ergodic theorem} states that
a transformation $T$ is ergodic if and only if for any $B\in\Sigma$,
\[
\lim_{N\to\infty}\int
\big{|}\frac{1}{N}\sum_{n=0}^{N-1}\bbone_{B}\circ T^{-n} - 
\mu(B)\big{|}^{2}d\mu 
= 0
\]
where $\bbone_{B}$ represents the indicator function of the set $B$.
A transformation $T$ is \textbf{totally ergodic} when for any 
$\ell\in\mathbb{N}$, the transformation $T^{\ell}$ is ergodic.
 A transformation $T$ is \textbf{mixing} when for any 
$A,B\in\Sigma$,
\[
\lim_{n\to\infty}\mu(T^{n}(A) \cap B) - \mu(A)\mu(B) = 0.
\]

\subsection{Spectral Measures}

For $T$ a transformation and $B\in\Sigma$ let $\sigma_{T,B}$ be the probability measure on $S^{1}$ (the unit circle) with Fourier coefficients, for $\ell \in \mathbb{Z}$,
\[
\widehat{\sigma_{T,B}}(\ell) = \frac{\mu(T^{\ell}(B) \cap B) - \mu(B)\mu(B)}{\mu(B)(1-\mu(B))}.
\]
\textbf{Spectral measure} shall mean such a measure.  The reader is referred to Chapters 9 and 11 of \cite{lemanbook} for the spectral theory of rank-one transformations.

Note that $T$ is ergodic if and only if for any spectral measure arising from $T$ we have 
\[
\int \big{|} \frac{1}{N}\sum_{n=1}^{N}z^{n}\big{|}^{2} d\sigma(z) \to 0
\]
by the mean ergodic theorem,
hence if and only if $\sigma(\{1\}) = 0$ (for $z \ne 1$ we have $\frac{1}{N}\sum_{n=0}^{N-1}z^{n} = \frac{1 - z^{N}}{N(1 - z)} \to 0$), and that $T$ is mixing if and only if $\widehat{\sigma}(n) \to 0$.  Other mixing-like properties behave similarly.

\subsection{Power Ergodicity}
Introduced in \cite{CS04} and \cite{CS10}, power ergodicity involves the 
power of a transformation being uniformly ergodic in the sense that 
the ergodic averages for each power of the transformation converge 
uniformly to zero. 

\begin{definition}
A transformation $T$ is \textbf{weakly power ergodic} when for any spectral measure $\sigma$ arising from $T$,
\[
\lim_{N\to\infty}\sup_{1 \leq k \leq N}\int\big{|}\frac{1}{N}\sum_{n=1}^{N}z^{nk}\big{|}^{2}d\sigma(z) = 0.
\]
\end{definition}
 This is called weak power ergodicity as there is also the variant where the supremum over $k$ runs over all nonzero integers but we will not need that property here.

\subsection{Dynamical Sequences}

The notion of partial sums of a sequence, introduced by the author and Silva in
\cite{CS04} 
in connection with showing mixing for rank-one transformations, involves a family of sequences generated from a given 
sequence as follows:

\begin{definition}
Let $\{a_{n}\}$ be a sequence and $k\in\mathbb{N}$.  The 
\textbf{$k^{th}$ partial sums} of
$\{a_{n}\}$, denoted $\{a_{n}^{(k)}\}$, are given by
\[
a_{n}^{(k)} = a_{n} + a_{n+1} + \ldots + a_{n+k-1} =
\sum_{z=0}^{k-1} a_{n+z}.
\]
\end{definition}

Dynamical sequences are the natural 
representation of the spacer sequences for rank-one transformations 
(see section \ref{S:rankones}).
\begin{definition}
A \textbf{dynamical sequence} $\{s_{n,j}\}_{\{r_{n}\}}$ is a 
doubly-indexed collection of nonnegative integers $s_{n,j}$ for 
$n\in\mathbb{N}$ and $j\in\{0,\ldots,r_{n}-1\}$ where $\{r_{n}\}$ is a 
sequence, called the index sequence of the dynamical sequence.
\end{definition}
\begin{definition}
Let $\{s_{n,j}\}_{\{r_{n}\}}$ be a dynamical sequence and 
$k\in\mathbb{N}$.  The \textbf{$k^{th}$ partial sums} of 
$\{s_{n,j}\}_{\{r_{n}\}}$, written $\{s_{n,j}^{(k)}\}_{\{r_{n}-k\}}$, are given by
\[
s_{n,j}^{(k)} = s_{n,j} + s_{n,j+1} + \cdots + s_{n,j+k-1} = \sum_{z=0}^{k-1}s_{n,j+z}
\]
\end{definition}
The motivation for this definition comes from looking at the spacers of rank-one transformations and the times at which levels intersect themselves nontrivially: for a level in the $n^{th}$ column, this will occur precisely at the time $kh_{n} + s_{n,j}^{(k)}$ for $0 \leq j+k < r_{n}$ (see Section \ref{S:rankones} for details).

\subsection{Mixing Properties on Sequences}

\begin{definition}
Let $T$ be a transformation.  A 
sequence $\{a_{n}\}$ is \textbf{ergodic} with respect 
to $T$ when for any spectral measure $\sigma$ arising from $T$,
\[
\lim_{N\to\infty}
\int\big{|}\frac{1}{N}\sum_{n=1}^{N}z^{a_{n}}\big{|}^{2}d\sigma(z) = 0.
\]
\end{definition}

The following definitions are the spectral versions of the definitions 
found in \cite{CS10} (stated for sequences rather than dynamical 
sequences as in that paper).

\begin{definition}
Let $T$ be a transformation.  A 
sequence $\{a_{n}\}$ is \textbf{totally ergodic} with respect 
to $T$ when for any fixed $k\in\mathbb{N}$ and any spectral measure $\sigma$ arising from $T$,
\[
\lim_{N\to\infty}
\int\big{|}\frac{1}{N}\sum_{n=1}^{N}z^{a_{n}^{(k)}}\big{|}^{2}d\sigma(z) = 0.
\]
\end{definition}

The choice of the term totally ergodic here comes from the fact that in the case when $a_{n} = n$ we have $a_{n}^{(k)} = kn + \nicefrac{k(k+1)}{2}$ and so the condition is exactly the usual notion of total ergodicity.

\begin{definition}
Let $T$ be a transformation.  A 
sequence $\{a_{n}\}$ is \textbf{weakly power ergodic} with 
respect to $T$ when for any spectral measure $\sigma$ arising from $T$,
\[
\lim_{N\to\infty}\sup_{1 \leq k \leq N}
\int\big{|}\frac{1}{N}\sum_{n=1}^{N}z^{a_{n}^{(k)}}\big{|}^{2}d\sigma(z) = 0.
\]
\end{definition}

These properties carry over to dynamical sequences in a straightforward 
manner by replacing $a_{n}$ in the above definitions with $s_{n,j}$ and 
the $N$ in the fraction and sum with $r_{n}$; 
this is how they were originally defined in \cite{CS10} and \cite{CS04}:
\begin{definition}
Let $T$ be a transformation and $\{s_{n,j}\}_{\{r_{n}\}}$ be a 
dynamical sequence.  The sequence  is \textbf{ergodic} with respect 
to $T$ when for any spectral measure $\sigma$ for $T$,
\[
\lim_{n\to\infty}
\int\big{|}\frac{1}{r_{n}}\sum_{j=0}^{r_{n}-1}z^{s_{n,j}}\big{|}^{2}d\sigma(z) = 0.
\]
The sequence  is \textbf{totally ergodic} with respect 
to $T$ when for any spectral measure $\sigma$ for $T$ and fixed $k$,
\[
\lim_{n\to\infty}
\int\big{|}\frac{1}{r_{n}}\sum_{j=0}^{r_{n}-1}z^{s_{n,j}^{(k)}}\big{|}^{2}d\sigma(z) = 0.
\]
The sequence  is \textbf{weakly power ergodic} with respect 
to $T$ when for any spectral measure $\sigma$ for $T$,
\[
\lim_{n\to\infty} \sup_{1 \leq k < r_{n}}
\int\big{|}\frac{1}{r_{n}}\sum_{j=0}^{r_{n}-1}z^{s_{n,j}^{(k)}}\big{|}^{2}d\sigma(z) = 0.
\]

\end{definition}

\begin{definition}
Let $T$ be a transformation.  A sequence $\{ a_{n} \}$ is \textbf{mixing} with respect to $T$ when for any spectral measure arising from $T$,
\[
\lim_{n \to \infty} \widehat{\sigma}(n) = 0
\]
\end{definition}

\section{Rank-One Transformations}\label{S:rankones}

We recall now the construction of rank-one transformations, in particular staircase transformations, and some of the results from \cite{CS04} and \cite{CS10} that we will use to prove mixing for stochastic staircase transformations.

The construction of rank-one transformations is by \textbf{cutting and stacking}.
Begin with $[0,1)$, the only \textbf{level} in the initial \textbf{column}.  \textbf{Cut} it into
$r_{0}$ \textbf{sublevels}, pieces of equal length: $[0,\frac{1}{r_{0}})$, 
$[\frac{1}{r_{0}},\frac{2}{r_{0}})$, $\ldots$, 
$[\frac{r_{0}-1}{r_{0}},1)$.
Place an interval of the same length above
$[\frac{1}{r_{0}},\frac{2}{r_{0}})$, i.e., place 
$[1,\frac{r_{0}+1}{r_{0}})$ above
$[\frac{1}{r_{0}},\frac{2}{r_{0}})$.  Likewise, place $j$ \textbf{spacer} sublevels 
above piece $j$.  Now, \textbf{stack} the resulting subcolumns
from left to right by placing $[0,\frac{1}{r_{0}})$ at the bottom, 
$[\frac{1}{r_{0}},\frac{2}{r_{0}})$ above it, the spacer level above that,
$[\frac{2}{r_{0}},\frac{3}{r_{0}})$ above
the spacer and so on, ending with the topmost 
of the $r_{0}-1$ spacers.  This stack of $h_{1} = r_{0} + 
\sum_{j=0}^{r_{0}-1}j$
levels (of length $\frac{1}{r_{0}}$), the second column, defines a map 
$T_{0} : 
[0,1+\frac{1}{r_{0}}\sum_{j=0}^{r_{0}-1}j-\frac{1}{r_{0}}) \to 
[\frac{1}{r_{0}},1+\frac{1}{r_{0}}\sum_{j=0}^{r_{0}-1}j)$ by 
sending points directly up one level.

Repeat the process: cut the entire new column into $r_{1}$ subcolumns 
of equal width $\frac{1}{r_{0}r_{1}}$, preserving the stack map on each 
subcolumn;
place $j$ spacers (intervals not yet in the space the same width as the subcolumns) above each subcolumn ($j\in\{0,\ldots,r_{1}-1\}$); and stack
the resulting subcolumns from left to right.  Our new column defines a map $T_{1}$ that agrees with $T_{0}$ where it is defined and extends it to all
but the topmost spacer of the rightmost subcolumn.  Iterating this process leads to a transformation $T$ defined on all but a Lebesgue measure zero set.

The transformations obtained in this manner (placing $j$ spacer levels above the $j^{th}$ subcolumn at each stage) are called \textbf{staircase transformations}.  More generally, one may place $s_{n,j}$ spacers above the $j^{th}$ subcolumn at the $n^{th}$ stage in place of the $j$ spacers above.  A transformation created by \emph{cutting and stacking} as just 
described (with a single column resulting from each iteration) is a \textbf{
rank-one transformation}.  Rank-one transformations are measurable and 
measure-preserving under Lebesgue measure, and are completely defined by the
dynamical sequence $\{s_{n,j}\}_{\{r_{n}\}}$ where at the $n^{th}$ step we cut
into $r_{n}$ pieces and place $s_{n,j}$ spacers above subcolumn $j$ (for staircase transformations, $s_{n,j} = j$).  
This $\{s_{n,j}\}_{\{r_{n}\}}$ is the \textbf{spacer sequence} for the 
transformation
and $\{r_{n}\}$ is the \textbf{cut sequence}.  The \textbf{height sequence} 
$\{h_{n}\}$ is the number of levels in each column: $h_{0} = 1$ and $h_{n+1} = r_{n}h_{n} + 
\sum_{j=0}^{r_{n}-1}s_{n,j}$.  We include a proof of the well-known fact:
\begin{proposition}\label{P:partrigid}
Let $T$ be a rank-one transformation with spacer sequence $\{s_{n,j}\}_{\{r_{n}\}}$ such that $\liminf r_{n} < \infty$.  Then $T$ is partially rigid hence cannot be mixing.
\end{proposition}
\begin{proof}
Let $R = \liminf r_{n}$ and let $B$ be a level in $C_{0}$ and set $t_{n} = h_{n} - s_{n,0}$.  Then for any $n$ such that $r_{n} \leq R+1$, it holds that $\mu(T^{t_{n}}(B) \cap B) \geq \frac{1}{R+1}\mu(B)$.
\end{proof}

Adams showed that 
a class of staircase transformations 
are mixing (\cite{Ad98}) and the author and Silva (\cite{CS10}) extended 
that result to all staircases and also showed that the class of 
\textbf{polynomial staircase transformations} (those with spacers 
$\{s_{n,j}\}_{\{r_{n}\}}$ given by $s_{n,j} = p_{n}(j)$ where $p_{n}$ 
are polynomials of bounded degree) are mixing.
Earlier, Ornstein \cite{Or72} had shown that if $s_{n,j+1} = x_{n,j+1} - x_{n,j}$ where the $x_{n,j}$ are uniform on $[-h_{n-1},\ldots,h_{n-1}]$ and independent and if $r_{n} \to \infty$ sufficiently fast then almost surely the resulting transformation is mixing (a more complete description of the construction, and a proof that such transformations are mixing, can also be found in \cite{Na98} pages 177--200).

\begin{theorem}[\cite{CS04} Theorem 3]\label{T:r1hmix}
Let $T$ be a rank-one transformation with spacer sequence $\{s_{n,j}\}_{\{r_{n}\}}$.  Then $\{h_{n}\}$ is mixing with respect to $T$ if $\{s_{n,j}\}_{\{r_{n}\}}$ is 
ergodic with 
respect to $T$.
\end{theorem}

\begin{corollary}\label{C:r1te}
Let $T$ be a rank-one transformation with spacer sequence $\{s_{n,j}\}_{\{r_{n}\}}$.  Then $T$ is totally ergodic if $\{s_{n,j}\}_{\{r_{n}\}}$ is 
ergodic with 
respect to $T$.
\end{corollary}
\begin{proof}
By Theorem \ref{T:r1hmix}, $T$ has a mixing sequence hence $T$ is weakly mixing.  Being weakly mixing, $1$ is the only eigenvalue of $T$ and it is simple hence $T$ is totally ergodic.
\end{proof}

\begin{theorem}[\cite{CS10}]\label{T:r1pe}
Let $T$ be a rank-one transformation with spacer sequence $\{s_{n,j}\}_{\{r_{n}\}}$.  Then $T$ is weakly power ergodic if 
$\{s_{n,j}\}_{\{r_{n}\}}$ is totally ergodic 
with respect to $T$.
\end{theorem}
\begin{proof}
In \cite{CS10} this is proven in full generality though stated only for staircases (in that paper weak power ergodicity is also stated for $\frac{k}{N} \to 0$ but the proof works for $k \leq N$ as we describe now).  Proposition 5.8 in \cite{CS10} states: let $T$ be a rank-one transformation with spacer sequence $\{s_{n,j}\}_{\{r_{n}\}}$ and $k \in \mathbb{N}$ such that $\{s_{n,j}^{(k)}\}_{\{r_{n}-k\}}$ is ergodic with respect to $T$, then $\{kh_{n}\}$ is mixing with respect to $T$ ($\{h_{n}\}$ being the height sequence for $T$).

The Block Lemma of \cite{Ad98} (see \cite{CS04} Lemma 7.1) states that
\[
\big{|}\frac{1}{N}\sum_{n=1}^{N}z^{n}\big{|} \leq \big{|}\frac{1}{L}\sum_{\ell=1}^{L}z^{\ell p}\big{|} + \frac{pL}{N}
\]
for any $p,L,N \in \mathbb{N}$.  For any $k \leq N$ we can choose $p$ and $y$ such that $h_{p} \leq k < h_{p+1} \leq ky < 2h_{p+1}$ (the choice of $p$ is uniquely determined by $k$ since $h_{n}$ is increasing; then choose $y$ minimally such that $ky \geq h_{p+1}$) and then
\[
\big{|}\frac{1}{N}\sum_{n=1}^{N} z^{nk}\big{|} \leq \big{|}\frac{1}{L}\sum_{\ell=1}^{L}z^{\ell ky}\big{|} + \frac{yL}{N}
\]
Now $\ell h_{p+1} \leq \ell ky < 2\ell h_{p+1}$ is a mixing sequence for each fixed $\ell$: Proposition 5.8 in \cite{CS10} shows that $\{ \ell h_{n} \}$ is mixing and that $\{ (\ell + 1)h_{n} \}$ is mixing; the claim then follows from Proposition 5.9 in \cite{CS10} (using $\ell$ as the $k(n)$ and $J_{n}$ being $0$ since $\nicefrac{s_{p,j}^{(\ell+1)}}{h_{p}} \to 0$).  Using the Blum-Hanson trick (\cite{CS04} Lemma 4.11) we then see that
\[
\int \big{|}\frac{1}{L}\sum_{\ell=1}^{L}z^{\ell k_{n}y_{n}}\big{|}^{2} d\sigma(z) = \frac{1}{L} + 2Re\Big{[}\frac{1}{L}\sum_{\ell=1}^{L-1}\frac{L-\ell}{L}\widehat{\sigma}(\ell k_{n}y_{n})\Big{]}
\]
and $\widehat{\sigma}(\ell k_{n}y_{n}) \to 0$ as $n \to \infty$ for each fixed $\ell$ so this term goes to zero.  We also get
\[
\frac{yL}{N} \leq \frac{2yL}{k} = \frac{2yk}{k^{2}} \leq \frac{2h_{p+1}}{h_{p}^{2}} \approx \frac{2r_{p}}{h_{p}} \to 0
\]
(since the measure space must be finite) and therefore
\[
\lim_{N\to\infty} \sup_{k \leq N} \int \big{|}\frac{1}{N}\sum_{n=1}^{N}z^{nk}\big{|} d\sigma(z) = 0
\]
which shows weak power ergodicity.
\end{proof}

\begin{theorem}[\cite{CS04} Theorem 6]\label{T:rgmix}
Let $T$ be a rank-one transformation with spacer sequence $\{s_{n,j}\}_{\{r_{n}\}}$ and height sequence $\{h_{n}\}$ such that $T$ has \textbf{\emph{restricted growth}}: $\frac{r_{n}^{2}}{h_{n}} \to 0$.  If $\{s_{n,j}\}_{\{r_{n}\}}$ is weakly power ergodic with respect to $T$ then $T$ is mixing.
\end{theorem} 

Note that what we call weak power ergodicity is referred to as uniform ergodicity in \cite{CS04}.  The reader familiar with \cite{CS04} and \cite{CS10} will be aware that removing the restricted growth condition for staircases is a nontrivial task.  The next theorem is needed for the non-restricted-growth case; it is essentially a consequence of Proposition 5.11 in \cite{CS10} (see also Definition 5.10 in \cite{CS10}).

\begin{theorem}\label{T:mix1}
Let $T$ be a rank-one transformation with spacer sequence $\{s_{n,j}\}_{\{r_{n}\}}$ and height sequence $\{h_{n}\}$.    Assume that for any sequence $\{ k_{n} \}$ such that $k_{n} < r_{n}$, any sequence $Q_{n} \to \infty$, any partition $\{\Gamma_{n,q}\}_{q \leq Q_{n}}$ of $\{0, \ldots, r_{n}-k_{n}-1\}$ such that $\frac{1}{Q_{n}}\sum_{q}\#\Gamma_{n,q} \to \infty$ and $h_{n}^{-1} |s_{n,j}^{(k_{n})} - s_{n,j^{\prime}}^{(k_{n})}| \to 0$ uniformly over $j,j^{\prime} \in \Gamma_{n,q}$ and any sequence $\{ \alpha_{n,q} \}$ with $\alpha_{n,q} \leq \alpha_{n,q+1}$ and $\frac{\alpha_{n,Q_{n}}}{k_{n}} \to 0$ that
\[
\lim_{N\to\infty} \int \frac{1}{r_{n}-k_{n}}\sum_{q=1}^{Q_{n}} \big{|}\sum_{j\in\Gamma_{n,q}} z^{s_{n,j}^{(k_{n}-\alpha_{n,q})}} \big{|} d\sigma(z) = 0.
\]
Then $T$ is mixing.
\end{theorem}
\begin{proof}
Let $t_{n} \to \infty$ be any sequence.  Proposition 5.11 in \cite{CS10} states that:

\begin{quote}
Set $p(n)$ and $k(n)$ 
such that $h_{p(n)} \leq t_{n} < h_{p(n)+1}$ and $k(n)h_{p(n)} \leq t_{n} 
< (k(n)+1)h_{p(n)}$.  If there exists a sequence $\{L_{n}\}$ such that $\frac{L_{n}}{h_{p_{n}}} \frac{1}{r_{p_{n}}}\sum_{j=0}^{r_{p_{n}}-1}s_{p_{n},j} \to 0$ and $L_{n} \to \infty$ that also has the property that every sequence of partitions that respects the spacing arrangement of $T$ with block sizes $\{\frac{h_{p_{n}}}{L_{n}}\}$ is \textbf{\emph{ergodic}} with respect to $T$, meaning
\[
\lim_{n\to\infty}\sum_{q=0}^{Q_{n}-1}\int\big{|}\frac{1}{r_{p(n)}}\sum_{j\in\Gamma_{n,q}}\chi_{B}\circ T^{-s_{p(n),j}^{(k(n)+2-\alpha(n, q))}}\big{|}d\mu = 0,
\]
then $\{t_{n}\}$ is mixing with respect to $T$.
\end{quote}

Definition 5.10 in \cite{CS10} states that

\begin{quote}
%Fix $T$ a rank-one transformation with spacer sequence $\{s_{n,j}\}_{\{r_{n}\}}$.
%Let $Q, p, k$ be positive integers and $\{ \Gamma_{q} \}_{q=0}^{Q-1}$ be a partition of $\mathbb{Z}_{r_{p}-k}$ (that is $\Gamma_{q} \subseteq \mathbb{Z}_{r_{p}-k}$ are disjoint and $\cup \Gamma_{q} = \mathbb{Z}_{r_{p}-k}$).  We say 
$(Q,p,k,\{\Gamma_{q}\})$ \textbf{\emph{respects the spacing arrangement at $k$}} of $T$ when $s_{p,j}^{(k)} \leq s_{p,j^{\prime}}^{(k)}$ for all $j \in \Gamma_{q}$ and $j^{\prime} \in \Gamma_{q^{\prime}}$ whenever $q \leq q^{\prime}$ and
%
%Given sequences $\{p_{n}\}$, $\{k_{n}\}$ and $\{Q_{n}\}$, 
a sequence of partitions $\{\Gamma_{n,q}\}_{Q_{n}}$ that each individually respect the spacing arrangement of $T$ at $k_{n}$ and also have the property that $\frac{Q_{n}}{r_{p_{n}}} \to 0$ is said to \textbf{\emph{respect the spacing arrangement}} for $T$ with \textbf{\emph{block sizes}} $\{b_{n}\}$ when $\#\Gamma_{n,q} \geq b_{n}$ for a density one set of $q$.
\end{quote}

The condition that $h_{n}^{-1} |s_{n,j}^{(k_{n})} - s_{n,j^{\prime}}^{(k_{n})}| \to 0$ uniformly over $j,j^{\prime} \in \Gamma_{n,q}$ ensures that the partitions respect the spacing arrangement.  The condition that $\frac{1}{Q_{n}}\sum_{q} \#\Gamma_{n,q} \to \infty$ provides the required $L_{n} = \frac{1}{Q_{n}}\sum_{q} \#\Gamma_{n,q}$.

The proposition then implies $\{ t_{n} \}$ is mixing, and as every increasing sequence is mixing this means $T$ is mixing.  Note: the converse is obviously true but we will not need this in the sequel.
\end{proof}

\section{Stochastic Staircase Transformations}

\subsection{Stochastically Generated Sequences}

We introduce the class of stochastically generated sequences and prove a series of facts about them, with the primary goal being Theorem \ref{T:P} stating that such sequences (almost surely) have a strong uniform ergodicity property.

\begin{definition}
Let $b$ an integer-valued random variable.  By Kolmogorov's Theorem there exists a probability space $(\Omega,P)$ with $b_{1},b_{2},\ldots : \Omega \to \mathbb{N}$ iid copies of $b$ and there exists a measurable map $\Theta : \Omega \to \Omega$ such that $b{n} \circ \Theta = b_{n+1}$ (usually referred to as a shift map).

Set $a_{n} = b_{1} + \cdots + b_{n}$.
Then $\{a_{n}\}$ is a (random) strictly increasing sequence in $\mathbb{N}$.  We call $\{a_{n}\} = \{a_{n}(\omega)\}$ a \textbf{stochastically generated sequence} and say that $\{a_{n}\}$ is \textbf{stochastically generated} by $b$ for a typical (generic) $\omega$.
\end{definition}

If we take $b$ to be a geometric variable with parameter $\alpha \in (0,1)$ (i.e. $P(b=n) = \alpha(1-\alpha)^{n-1}$ for $n\in\mathbb{N}$) then $\{a_{n}\}$ is a random sequence where each $m \in \{a_{n}\}$ (i.e. there exists $n$ such that $m = a_{n}$) independently with probability $\alpha$.  Such a sequence is often called a \textbf{randomly generated sequence} with density $\alpha$.

\begin{definition}
For an integer-valued random variable $b$, the \textbf{period} of $b$ is the largest $\ell \in \mathbb{N}$ such that $P(\ell \text{ divides } b) = 1$.  If $\ell = 1$ then $b$ is \textbf{aperiodic}.
\end{definition}

\subsection{Stochastically Generated Dynamical Sequences}

\begin{definition}
Let $\{b_{n}\}$ a sequence of (not necessarily iid) integer-valued random variables.  Let $(\Omega,P)$ be a probability space such that $\{b_{n,j}\}_{n,j\in\mathbb{N}}$ are independent random variables (independent over $j$; over $n$ they may interact) on $(\Omega,P)$ where the $b_{n,j}$ are iid copies of $b_{n}$ for each $n$.  Let $\{ r_{n} \}$ be a sequence of positive integers such that $r_{n} \to \infty$.
Set $s_{n,j} = b_{n,1} + \cdots + b_{n,j}$.  Then a typical $\{s_{n,j}\}_{\{r_{n}\}}$ is a \textbf{stochastically generated dynamical sequence} generated by $\{ b_{n} \}$ with cut sequence $\{ r_{n} \}$. 
\end{definition}

We will use the notation $\mathbb{E}$ to represent the expectation functional on $(\Omega,P)$ and reserve $\int \cdot~d\sigma$ for the integral functional with respect to the spectral measures.

\subsection{Stochastic Staircase Transformations}

\begin{definition}
Let $b$ be an aperiodic positive-integer-valued random variable with finite mean.  Let $\{ a_{n} \}$ be a sequence stochastically generated by $b$ and let $\{ r_{n} \}$ be a sequence of positive integers such that $r_{n} \to \infty$.  The rank-one transformation with spacer sequence $\{ s_{n,j} \}_{\{r_{n}\}}$ where $s_{n,j} = a_{j}$ is a \textbf{stochastic staircase transformation} generated by $b$ with cut sequence $\{ r_{n} \}$.
\end{definition}

If the random variable that generates the stochastic sequence is taken to be a geometric variable (so that the spacer sequence is randomly generated sequence) then the resulting transformation is called a \textbf{random staircase transformation}.  If the generating random variable is taken to be uniform the resulting transformations are similar to Ornstein's transformations (see section \ref{S:ornstein}).  If the generating random variables is taken to be identically one then the resulting transformations are simply the classical (non-stochastic) staircase transformations.

More generally, one can allow the base variable $b$ to take on negative values:
\begin{definition}
Let $b$ be an aperiodic integer-valued random variable with finite mean.  Let $\{ a_{n} \}$ be a sequence stochastically generated by $b$ and let $\{ r_{n} \}$ be a sequence of positive integers such that $r_{n} \to \infty$.  Let $\{ x_{n} \}$ be a sequence of nonnegative integers such that $x_{n} \geq - \inf_{\Omega} \inf_{j \leq r_{n}} a_{j}$.  Note that most often we will take $b \geq 0$ so $x_{n}$ may be taken to be zero.
The rank-one transformation with spacer sequence $\{ s_{n,j} \}_{\{r_{n}\}}$ where $s_{n,j} = a_{j} + x_{n}$ is a \textbf{stochastic staircase transformation} generated by $b$ with cut sequence $\{ r_{n} \}$ (and padding $\{ x_{n} \}$).
\end{definition}

In fact, one can use a triangular array of random variables, and even allow the cut parameter to be random (as long as it depends only on the previous steps):
\begin{definition}
More generally, let $\{ b_{n} \}$ be a sequence of (not necessarily iid) aperiodic integer-values random variables with finite mean (not necessarily uniformly bounded over $n$).  Let $\{ a_{n,j} \}$ be a dynamical sequence stochastically generated by the $b_{n}$ and let $\{ r_{n} \}$ be a sequence of positive integers such that each $r_{n}$ is a function of the $a_{m,j}$ for $m < n$ and such that $\mathbb{E}b_{n}$ is bounded by some polynomial in $r_{n}$.  The rank-one transformation with spacer sequence $\{ s_{n,j} \}_{\{r_{n}\}}$ where $s_{n,j} = a_{n,j} + x_{n}$ (where $\{ x_{n} \}$ is as above) is also referred to as a \textbf{stochastic staircase transformation} generated by $\{ b_{n} \}$ with cut sequence $\{ r_{n} \}$.
\end{definition}

Since we are concerned with mixing, we will need to ensure that when the cut sequence is random that it tends to infinity (see Proposition \ref{P:partrigid}).  We will also be concerned with ensuring that the transformations are defined on finite measure-spaces.  To this end, since the $\{ r_{n} \}$ are chosen without reference to the random variables:
\begin{proposition}

A stochastic staircase transformation with spacer sequence $\{s_{n,j}\}_{\{r_{n}\}}$ will be defined on a finite measure space a.s.~provided that
\[
\sum_{n} \Big{(} \frac{1}{r_{n}} \prod_{j=1}^{n-1}r_{j} \Big{)}^{-1} < \infty
\]
\end{proposition}
\begin{proof}
Proposition 8.1 in \cite{CS10} shows that the space being finite measure is equivalent to the condition that
\[
\sum_{n=1}^{\infty} \frac{\overline{s_{n}}}{h_{n}} < \infty
\]
where $\overline{s_{n}} = \frac{1}{r_{n}}\sum_{j=0}^{r_{n}-1} s_{n,j}$ is the average spacer height.  In the context of stochastic staircase transformations,
\[
\overline{s_{n}} = b_{1} + \frac{r_{n}-1}{r_{n}} b_{2} + \cdots + \frac{1}{r_{n}}b_{n}
\]
and since
\[
\sum_{j=0}^{r_{n}-1} \frac{1}{(j+1)^{2}} \textrm{Var}(\frac{r_{n}-j}{r_{n}}b_{j}) = \sum_{j=0}^{r_{n}-1} \frac{1}{(j+1)^{2}} \frac{(r_{n}-j)^{2}}{r_{n}^{2}} \textrm{Var}(b) \leq \sum_{j=0}^{r_{n}-1} \frac{1}{(j+1)^{2}} \textrm{Var}(b) \leq \frac{\pi^{2}}{6} \textrm{Var}(b) < \infty
\]
by Kolmogorov's Weighted Strong Law of Large Numbers (\cite{SS} Theorem 2.3.10),
\[
\lim_{n} \frac{\overline{s_{n}}}{r_{n}} = \mathbb{E}b
\]
almost surely.  Now
\[
h_{n+1} = r_{n}h_{n} + r_{n}\overline{s_{n}} \approx r_{n}h_{n} + r_{n}^{2}\mathbb{E}b = r_{n}h_{n}(1 + \frac{r_{n}}{h_{n}}\mathbb{E}b)
\]
and since $\frac{r_{n}}{h_{n}} \to 0$ this means $h_{n+1} \approx r_{n}h_{n}$
and therefore
$h_{n} \approx \prod_{j<n} r_{j}$
and so
\[
\sum \frac{\overline{s_{n}}}{h_{n}} = \sum \frac{\overline{s_{n}}}{r_{n}} \frac{r_{n}}{h_{n}} 
\approx (\mathbb{E}b) \sum \frac{r_{n}}{h_{n}} \approx (\mathbb{E}b) \sum \Big{(}\frac{1}{r_{n}}\prod_{j=1}^{n-1} r_{j}\Big{)}^{-1}
\]
\end{proof}

\section{Properties of Stochastically Generated Sequences}

\subsection{Basic Technical Properties}

\begin{lemma}\label{L:remarks}
For any $n,k,t$ and stochastic sequence $\{ a_{n} \}$, letting $\Theta$ be the shift and $a_{0} = 0$,
\begin{align*}
(i) \quad&a_{n+t}^{(k)} - a_{n}^{(k)} = a_{n+k}^{(t)} - a_{n}^{(t)}; \\
(ii) \quad&a_{n}^{(k)} = ka_{n} + a_{0}^{(k)} \circ \Theta^{n}; \\
(iii) \quad&a_{n+t} - a_{n} = a_{t} \circ \Theta^{n}; and \\
(iv) \quad&a_{n+t}^{(k)} - a_{n}^{(k)} = (a_{t}^{(k)} - a_{0}^{(k)}) \circ \Theta^{n}.
\end{align*}
\end{lemma}
\begin{proof}
\begin{align*}
a_{n+t}^{(k)} - a_{n}^{(k)} &= \sum_{j=0}^{k-1} a_{n+t+j} - \sum_{j=0}^{k-1} a_{n+j} \\
&= a_{n+t} + \cdots + a_{n+k+t-1} - a_{n} - \cdots - a_{n+k-1} \\
&= a_{n+k} + \cdots + a_{n+k+t-1} - a_{n} - \cdots - a_{n+t-1} \\
&= \sum_{j=0}^{t-1} a_{n+k+j} - \sum_{j=0}^{t-1} a_{n+j} = a_{n+k}^{(t)} - a_{n}^{(t)}
\end{align*}
and
\[
a_{k+t} - a_{t} = \sum_{j=1}^{k+t} b_{j} - \sum_{j=1}^{t} b_{j} = \sum_{j=t+1}^{t+k} b_{j} = \sum_{j=1}^{k} b_{j}\circ \Theta^{t} = a_{k} \circ \Theta^{t}
\]
and so
\[
a_{n}^{(k)} = \sum_{j=0}^{k-1} a_{n+j} = \sum_{j=0}^{k-1} (a_{n+j} - a_{n} + a_{n}) = ka_{n} + \sum_{j=0}^{k-1} a_{j}\circ \Theta^{n} = ka_{n} + a_{0}^{(k)} \circ\Theta^{n}
\]
hence
\begin{align*}
a_{n+t}^{(k)} - a_{n}^{(k)} &= ka_{n+t} + a_{0}^{(k)} \circ\Theta^{n+t} - ka_{n} - a_{0}^{(k)}\circ\Theta^{n} \\
&= (ka_{t} + a_{0}^{(k)}\circ\Theta^{t} - a_{0}^{(k)}) \circ\Theta^{n} = (a_{t}^{(k)} - a_{0}^{(k)}) \circ\Theta^{n}.
\end{align*}
\end{proof}

\begin{lemma}\label{L:slln}
Let $\{ b_{m} \}_{m \in \mathbb{N}}$ be an iid sequence of random variables and $\{ X_{n} \}_{n\in\mathbb{N}}$ be random variables depending only on the $b_{m}$ such that there exist constants $B,D,K \geq 1$ where
\begin{align*}
(i) \quad&\mathbb{E}X_{n} = 0 \text{ for each $n$};\\
(ii) \quad&|X_{n}| \leq D \text{ almost surely for each $n$}; \\
(iii) \quad&\#\{ m : X_{n} \text{ depends on } b_{m} \} \leq K \text{ for each $n$}; and \\
(iv) \quad&\#\{ n : X_{n} \text{ depends on } b_{m} \} \leq B \text{ for each $m$}.
\end{align*}
Then almost surely
\[
\lim_{N\to\infty} \frac{1}{N}\sum_{n=1}^{N} X_{n} = 0.
\]
\end{lemma}
\begin{proof}
Observe first that for a fixed $n$, by assumption there are at most $K$ choices of $m$ such that $X_{n}$ depends on $b_{m}$.  For each $b_{m}$ there are at most $B$ choices of $q$ such that $X_{q}$ depends on $b_{m}$.  Therefore for each $n$ there are at most $B^{K}$ choices of $q$ such that $X_{q}$ is not independent of $X_{n}$.  As the $X_{n}$ are mean zero, if $X_{n}$ and $X_{q}$ are independent then $\mathbb{E}[X_{n}X_{q}] = 0$, and so
\[
\mathbb{E}\Big{[}\big{|}\frac{1}{N}\sum_{n=1}^{N} X_{n}\big{|}^{2}\Big{]} = \frac{1}{N^{2}}\sum_{n,q=1}^{N}\mathbb{E}[X_{n}X_{q}] \leq \frac{B^{K}}{N^{2}}D^{2}.
\]
Then
\[
\mathbb{E}\Big{[}\sum_{N=1}^{\infty} \big{|}\frac{1}{N}\sum_{n=1}^{N} X_{n}\big{|}^{2}\Big{]} \leq B^{K} D^{2} \sum_{N=1}^{\infty} \frac{1}{N^{2}} = B^{K} D^{2} \frac{\pi^{2}}{6}
\]
and so almost surely
\[
\sum_{N=1}^{\infty} \big{|}\frac{1}{N}\sum_{n=1}^{N} X_{n}\big{|}^{2} < \infty
\]
and, in particular,
\[
\lim_{N \to \infty} \frac{1}{N}\sum_{n=1}^{N} X_{n} = 0 \quad\text{a.s.}
\]
\end{proof}

We will need a more concrete version of the previous lemma as well:
\begin{lemma}\label{L:key}
Let $\{ b_{m} \}_{m \in \mathbb{N}}$ be an iid sequence of random variables and $\{ X_{n} \}_{n\in\mathbb{N}}$ be random variables depending only on the $b_{m}$ such that there exist constants $B,D,K \geq 1$ where
\begin{align*}
(i) \quad&\mathbb{E}X_{n} = 0 \text{ for each $n$};\\
(ii) \quad&|X_{n}| \leq D \text{ almost surely for each $n$}; \\
(iii) \quad&\#\{ m : X_{n} \text{ depends on } b_{m} \} \leq K \text{ for each $n$}; and \\
(iv) \quad&\#\{ n : X_{n} \text{ depends on } b_{m} \} \leq B \text{ for each $m$}.
\end{align*}
Then for any $N,R \in \mathbb{N}$ and $\delta > 0$
\[
P(\big{|}\frac{1}{N}\sum_{n=1}^{N}X_{n}\big{|} \geq \delta) \leq C_{R}\delta^{-2R}(DKB)^{2R}N^{-R}
\]
where $C_{R}$ is a constant depending only on $R$.
\end{lemma}
\begin{proof}
Set
$Z_{n} = \{ m : X_{n} \text{ depends on } b_{m} \}$
so that $\# Z_{n} \leq K$ by hypothesis.  Note also that for each $m$,
$\# \{ n : m \in Z_{n} \} \leq B$.
Let
\begin{align*}
\mathcal{Q} = \big{\{} (n_{1},&\ldots,n_{2R}) \in \{ 1, \ldots, N \}^{2R} : \\
&\forall j \in \{ 1, \ldots, 2R \} \text{ } \exists j^{\prime} \in \{ 1, \ldots, 2R \}, j \ne j^{\prime}, Z_{n_{j}} \cap Z_{n_{j^{\prime}}} \ne \emptyset \big{\}}
\end{align*}
and set $Q = \# \mathcal{Q}$.

By the Chebyshev Inequality idea,
\begin{align*}
P(\big{|}\frac{1}{N}\sum_{n=1}^{N}X_{n}\big{|} \geq \delta) &= \mathbb{E}\bbone_{\delta^{-1} \big{|}\frac{1}{N}\sum_{n=1}^{N}X_{n}\big{|} \geq 1} \\
&\leq \mathbb{E}\Big{|}\delta^{-1} \big{|}\frac{1}{N}\sum_{n=1}^{N}X_{n}\big{|}\Big{|}^{2R} \\
&= \delta^{-2R} \mathbb{E}\big{|}\frac{1}{N}\sum_{n=1}^{N}X_{n}\big{|}^{2R} \\
&= \delta^{-2R} N^{-2R} \sum_{(n_{1},\ldots,n_{2R})\in\{1,\ldots,N\}^{2R}} \mathbb{E}\prod_{j=1}^{2R}X_{n_{j}} \\
&= \delta^{-2R} N^{-2R} \Big{[}  \sum_{(n_{1},\ldots,n_{2R})\in\mathcal{Q}} \mathbb{E}\prod_{j=1}^{2R}X_{n_{j}} + \sum_{(n_{1},\ldots,n_{2R})\notin\mathcal{Q}} 0 \Big{]} \\
&\leq \delta^{-2R} N^{-2R} Q D^{2R}
\end{align*}
where the next to last line follows since for $(n_{1},\ldots,n_{2R}) \notin \mathcal{Q}$ there is some isolate $Z_{n_{j}}$ (in the sense that $X_{n_{j}}$ is independent of the other $X_{n_{j^{\prime}}}$) giving an $\mathbb{E}X_{n_{j}} = 0$ factor in the product.
Let
\[
\mathcal{P} = \{ \text{ $p$ a partition of } \{ 1, \ldots, 2R \} : \text{ no element is alone } \big{\}}
\]
and observe that $\mathcal{P}$ depends only on $R$.  Our plan is to split the $n_{j}$ into the collections where there is overlap among the corresponding $Z_{n_{j}}$ (a partition in $\mathcal{P}$) and count the number of possibilities for the overlap from there.
For a given partition $p \in \mathcal{P}$ and $q \in p$ let
\[
\mathcal{Q}_{p,q} = \big{\{} (n_{j})_{j \in q} \in \{1,\ldots,N\}^{\#q} : \nexists j \in q \text{ } \forall j^{\prime} \in q, j \ne j^{\prime}, Z_{n_{j}} \cap Z_{n_{j^{\prime}}} = \emptyset \big{\}}
\]
be the set of $\#q$-tuples where all the corresponding $Z_{n_{j}}$ interact with one another.  Observe that
\[
\# \mathcal{Q}_{p,q} \leq (\#q)! \# \big{\{} (n_{1},\ldots,n_{\#q}) \in \{1,\ldots,N\}^{\#q} : \forall j \in \{1,\ldots,N\} Z_{n_{j}} \cap Z_{n_{j+1}} \ne \emptyset \big{\}}
\]
since we can rearrange the $j$ (that is, assuming say $1 \in q$ we know that $Z_{n_{1}}$ must interact with some $Z_{n_{j}}$, $j\ne1$ which in turn must interact with some $Z_{n_{j^{\prime}}}$, $j^{\prime} \ne j,1$ and this can be continued without a cycle since the $Z_{n_{j}}$ all interact in $q$).  Therefore
\[
\# \mathcal{Q}_{p,q} \leq (\#q)! NK (BK)^{\#q - 1}
\]
by choosing $n_{1}$ to be any of $N$ choices, choosing a coordinate (at most $K$ choices) such that $Z_{n_{1}} \cap Z_{n_{2}} \ne \emptyset$ is witnessed by that coordinate, choosing $n_{2}$ to be one of the at most $B$ choices where such intersection is possible and repeating for the remaining $n_{j}$.
Now
\begin{align*}
Q &= \#\mathcal{Q} = \sum_{p \in \mathcal{P}} \prod_{q \in p} \# \mathcal{Q}_{p,q} 
\leq \sum_{p \in \mathcal{P}} \prod_{q \in p} \big{(}(\#q)! NK (BK)^{\#q - 1}\big{)} \\
&\leq \sum_{p \in \mathcal{P}} \big{(}\prod_{q \in p}(\#q)!\big{)} N^{\#p} (BK)^{\sum_{q\in p}\#q}
\leq \sum_{p \in \mathcal{P}} (2R)! N^{R} (BK)^{2R} 
= (\#\mathcal{P})(2R)! (BK)^{2R} N^{R}
\end{align*}
since the size of a partition $p \in \mathcal{P}$ is at most $R$ (which occurs when the $2R$ elements are partitioned into pairs since no element can be alone).
Therefore
\begin{align*}
P(\big{|}\frac{1}{N}\sum_{n=1}^{N}\big{|} \geq \delta)
&\leq \delta^{-2R} N^{-2R} D^{2R} (\#\mathcal{P})(2R)! (BK)^{2R} N^{R} \\
&= (\delta^{-1}DBK)^{2R}C_{R} N^{-R}
\end{align*}
where $C_{R} = (\#\mathcal{P})(2R)!$ depends only on $R$.
\end{proof}

\begin{lemma}\label{L:key2single}
Let $\{ b_{m} \}_{m\in\mathbb{N}}$ be an iid sequence of random variables and $\{ Y_{n} \}_{n \in \mathbb{N}}$ be a sequence of $\mathbb{N}$-valued random variables depending only on the $b_{m}$ such that there exist constants $B,K,L \geq 1$ where
\begin{align*}
(i) \quad&\#\{ m : Y_{n} \text{ depends on } b_{m} \} \leq K \text{ for each $n$};\\
(ii) \quad&\#\{ n : Y_{n} \text{ depends on } b_{m} \} \leq B \text{ for each $m$}; and \\
(iii) \quad&\limsup_{N\to\infty} \frac{1}{N}\sum_{n=1}^{N}Y_{n} + \mathbb{E}Y_{n} \leq L \text{ almost surely}.
\end{align*}
Then there exists a measure one set on which for every $z \in S^{1}$
\[
\lim_{N \to \infty} \frac{1}{N}\sum_{n=1}^{N}z^{Y_{n}} - \mathbb{E}\frac{1}{N}\sum_{n=1}^{N}z^{Y_{n}} = 0.
\]
\end{lemma}

 This is a trivial consequence of the following:
\begin{lemma}\label{L:key2}
Let $\{ b_{m} \}_{m\in\mathbb{N}}$ be an iid sequence of random variables and $\{ Y_{n}^{(k)} \}_{n,k \in \mathbb{N}}$ be a family of sequences of $\mathbb{Z}$-valued random variables depending only on the $b_{m}$ such that there exist constants $B,K,L \geq 1$ where
\begin{align*}
(i) \quad&\#\{ m : Y_{n}^{(k)} \text{ depends on } b_{m} \} \leq K \text{ for each $n$ and $k$};\\
(ii) \quad&\#\{ n : Y_{n}^{(k)} \text{ depends on } b_{m} \} \leq B \text{ for each $m$ and $k$}; and \\
(iii) \quad&\limsup_{N\to\infty} \sup_{k \leq N} \frac{1}{N}\sum_{n=1}^{N}|Y_{n}^{(k)}| + \mathbb{E}|Y_{n}^{(k)}| \leq L \text{ almost surely}.
\end{align*}
Then there exists a measure one set on which for every $z \in S^{1}$
\[
\lim_{N \to \infty} \sup_{k \leq N} \Big{|}\frac{1}{N}\sum_{n=1}^{N}z^{Y_{n}^{(k)}} - \mathbb{E}\frac{1}{N}\sum_{n=1}^{N}z^{Y_{n}^{(k)}} \Big{|} = 0.
\]
\end{lemma}
\begin{proof}
Consider the functions
\[
X_{n}^{(k)}(z) = z^{Y_{n}^{(k)}} - \mathbb{E}z^{Y_{n}^{(k)}}.
\]
Then $\mathbb{E}X_{n}^{(k)}(z) = 0$ and $|X_{n}^{(k)}(z)| \leq 2$ so we may apply Lemma \ref{L:key} for each fixed $k$ and $z$ and obtain that for any $N,R \in \mathbb{N}$ and $\delta > 0$
\[
P(\big{|}\frac{1}{N}\sum_{n=1}^{N}X_{n}^{(k)}(z)\big{|} \geq \delta) \leq C_{R}\delta^{-2R}(2KB)^{2R}N^{-R}.
\]
Hence, setting $R = 3$,
\begin{align*}
P(\sup_{k \leq N} \big{|}\frac{1}{N}\sum_{n=1}^{N}X_{n}^{(k)}(z)\big{|} \geq \delta) 
&\leq \sum_{k=1}^{N} P(\big{|}\frac{1}{N}\sum_{n=1}^{N}X_{n}^{(k)}(z)\big{|} \geq \delta) \\
&\leq \sum_{k=1}^{N} C_{3}\delta^{-6}(2KB)^{6}N^{-3} \\
&= C_{3}\delta^{-6}(2KB)^{6} N^{-2}.
\end{align*}
Note that for any $z,w \in S^{1}$ and $t \in \mathbb{N}$ we have that
\[
|z^{t} - w^{t}| \leq t |z - w|
\]
since the $t = 1$ case is immediate and
\[
|z^{t+1} - w^{t+1}| \leq |z^{t} - w^{t}||z| + |z - w||w^{t}| \leq |z^{t} - w^{t}| + |z - w|.
\]
Therefore
\[
|X_{n}^{(k)}(z) - X_{n}^{(k)}(w)| \leq |z^{Y_{n}^{(k)}} - w^{Y_{n}^{(k)}}| + \mathbb{E}|z^{Y_{n}^{(k)}} - w^{Y_{n}^{(k)}}|
\leq \big{(}|Y_{n}^{(k)}| + \mathbb{E}|Y_{n}^{(k)}|\big{)} |z - w|
\]
and so
\[
\big{|}\frac{1}{N}\sum_{n=1}^{N}X_{n}^{(k)}(z) - \frac{1}{N}\sum_{n=1}^{N}X_{n}^{(k)}(w)\big{|}
\leq \Big{(}\frac{1}{N}\sum_{n=1}^{N}|Y_{n}^{(k)}| + \mathbb{E}|Y_{n}^{(k)}|\Big{)}|z - w|.
\]
Let $z_{j} \in S^{1}$ for $j=1,\ldots,J$ be a set of (irrational) points such that for any $z \in S^{1}$ we have $\sup_{j} |z - z_{j}| < \frac{\delta}{L + 1}$ so we may take $J \leq \frac{L + 1}{\delta} + 1$.  Observe that if for some $z \in S^{1}$,
\[
\big{|}\frac{1}{N}\sum_{n=1}^{N}X_{n}^{(k)}(z)\big{|} \geq 2\delta
\]
then either there is some $j \in \{ 1, \ldots, J \}$ such that
\[
\big{|}\frac{1}{N}\sum_{n=1}^{N}X_{n}^{(k)}(z_{j})\big{|} \geq \delta
\]
or else it must be that
\[
\frac{1}{N}\sum_{n=1}^{N}|Y_{n}^{(k)}| + \mathbb{E}|Y_{n}^{(k)}| \geq L + 1
\]
since $|z - z_{j}| < \frac{\delta}{L + 1}$.  Hence if for some $z$,
\[
\limsup_{N \to \infty} \sup_{k \leq N} \big{|}\frac{1}{N}\sum_{n=1}^{N}X_{n}^{(k)}(z)\big{|} \geq 2\delta
\]
then either for some $j$ (always choose the $j$ such that $z_{j}$ is closest to $z$)
\[
\limsup_{N \to \infty} \sup_{k\leq N}\big{|}\frac{1}{N}\sum_{n=1}^{N}X_{n}^{(k)}(z_{j})\big{|} \geq \delta
\]
or else
\[
\limsup_{N \to \infty} \sup_{k\leq N}\frac{1}{N}\sum_{n=1}^{N}|Y_{n}^{(k)}| + \mathbb{E}|Y_{n}^{(k)}| \geq L + 1
\]
since a sequence of $N$ such that the first limit is $\geq 2\delta$ gives a subsequence where one of the other two is $\geq \delta$ or $L + 1$.
Therefore
\begin{align*}
&P(\sup_{z \in S^{1}} \limsup_{N \to \infty} \sup_{k \leq N} \big{|}\frac{1}{N}\sum_{n=1}^{N}X_{n}^{(k)}(z)\big{|} \geq 2\delta) \\
&\quad\quad\quad\quad\quad\quad\leq P(\limsup_{N \to \infty} \sup_{k \leq N} \frac{1}{N}\sum_{n=1}^{N}|Y_{n}^{(k)}| + \mathbb{E}|Y_{n}^{(k)}| \geq L + 1) \\
&\quad\quad\quad\quad\quad\quad\quad\quad+ \sum_{j=1}^{J} P(\limsup_{N \to \infty} \sup_{k \leq N} \big{|}\frac{1}{N}\sum_{n=1}^{N}X_{n}^{(k)}(z_{j})\big{|} \geq \delta).
\end{align*}
Now by hypothesis, $\limsup_{N \to \infty} \sup_{k \leq N} \frac{1}{N}\sum_{n=1}^{N}|Y_{n}^{(k)}| + \mathbb{E}|Y_{n}^{(k)}| \leq L$ almost surely so the first probability is zero.  Using the Borel-Cantelli idea,
\begin{align*}
P(\limsup_{N \to \infty} \sup_{k \leq N} &\big{|}\frac{1}{N}\sum_{n=1}^{N}X_{n}^{(k)}(z_{j})\big{|} \geq \delta) \\
&= P(\bigcap_{M=1}^{\infty}\bigcup_{N=M}^{\infty} \sup_{k \leq N} \big{|}\frac{1}{N}\sum_{n=1}^{N}X_{n}^{(k)}(z_{j})\big{|} \geq \delta) \\
&\leq \lim_{M\to\infty} \sum_{N=M}^{\infty} P(\sup_{k \leq N} \big{|}\frac{1}{N}\sum_{n=1}^{N}X_{n}^{(k)}(z_{j})\big{|} \geq \delta) \\
&\leq \lim_{M\to\infty} \sum_{N=M}^{\infty} C_{3}\delta^{-6}(2KB)^{6}N^{-2} = 0
\end{align*}
and since $J$ is fixed we obtain that
\[
P(\sup_{z \in S^{1}} \limsup_{N \to \infty} \sup_{k \leq N} \big{|}\frac{1}{N}\sum_{n=1}^{N}X_{n}^{(k)}(z)\big{|} \geq 2\delta) = 0.
\]
Taking the measure one set for each rational $\delta > 0$ and unioning gives that
\[
P(\sup_{z \in S^{1}} \limsup_{N \to \infty} \sup_{k \leq N} \big{|}\frac{1}{N}\sum_{n=1}^{N}X_{n}^{(k)}(z)\big{|} > 0) = 0.
\]
Hence there is a measure one set on which for any $z \in S^{1}$
\[
\lim_{N \to \infty} \sup_{k\leq N}\big{|}\frac{1}{N}\sum_{n=1}^{N}X_{n}^{(k)}(z)\big{|} = 0.
\]
Plugging back in the definition of $X_{n}^{(k)}(z)$ this means that on this measure one set for every $z \in S^{1}$,
\[
\lim_{N \to \infty} \sup_{k \leq N} \Big{|}\frac{1}{N}\sum_{n=1}^{N}z^{Y_{n}^{(k)}} - \mathbb{E}\frac{1}{N}\sum_{n=1}^{N}z^{Y_{n}^{(k)}}\Big{|} = 0.
\]
\end{proof}

\subsection{The Strong Law for Stochastic Sequences}

\begin{proposition}\label{P:ergodic}
Let $\{ a_{n} \}$ be a stochastic sequence generated by a random variable $b$ with finite mean.  Let $\ell$ be a fixed positive integer.  Then there is a measure one set on which for every $z \in S^{1}$,
\[
\lim_{N \to \infty} \frac{1}{N}\sum_{n=1}^{N}z^{a_{n+\ell} - a_{n}} = \mathbb{E}z^{a_{\ell}}.
\]
\end{proposition}
 This is an easy consequence of the following (with $k = 1$):
\begin{proposition}\label{P:totallyergodic}
Let $\{ a_{n} \}$ be a stochastic sequence generated by a random variable $b$ with finite mean.  Let $\ell, k$ be fixed positive integers.  Then there is a measure one set on which for every $z \in S^{1}$,
\[
\lim_{N \to \infty} \frac{1}{N}\sum_{n=1}^{N}z^{a_{\ell+n}^{(k)} - a_{n}^{(k)}} = \mathbb{E}z^{a_{\ell}^{(k)} - a_{0}^{(k)}}.
\]
\end{proposition}
\begin{proof}
Let
\[
Y_{n} = a_{n+\ell}^{(k)} - a_{n}^{(k)} = \big{(}a_{\ell}^{(k)} - a_{0}^{(k)}\big{)} \circ \Theta^{n}
\]
which depends only on coordinates $b_{n+1}, \ldots, b_{n+\ell+k}$ and observe that
\[
\#\{ m : Y_{n} \text{ depends on } b_{m} \} \leq \ell + k \quad\quad \text{ for each $n$}
\]
and
\[
\#\{ n : Y_{n} \text{ depends on } b_{m} \} \leq \ell + k \quad\quad \text{ for each $m$}.
\]
Now for each $t$
\[
\lim_{N \to \infty} \frac{1}{N}\sum_{n=1}^{N}a_{t} \circ \Theta^{n}
= \sum_{j=1}^{t} \lim_{N \to \infty} \frac{1}{N}\sum_{n=1}^{N}b_{j} \circ \Theta^{n}
= \sum_{j=1}^{t} \mathbb{E}b = t \mathbb{E}b
\]
almost surely by the (usual) Strong Law of Large Numbers.  Likewise $\mathbb{E}a_{t} = t \mathbb{E}b$.  Therefore
\begin{align*}
\lim_{N \to \infty} \frac{1}{N}\sum_{n=1}^{N}a_{\ell}^{(k)}\circ\Theta^{n} + \mathbb{E}a_{\ell}^{(k)}\circ \Theta^{n}
&= \sum_{t=0}^{k-1} \lim_{N \to \infty} \frac{1}{N}\sum_{n=1}^{N}a_{\ell+t}\circ\Theta^{n} + \mathbb{E}a_{\ell+t}\\
&= \sum_{t=0}^{k-1} 2(\ell + t)\mathbb{E}b = (2k\ell + k(k+1))\mathbb{E}b
\end{align*}
almost surely (take the union of the measure one sets for each of the countably many values of $t \in \mathbb{N}$).
Hence by Lemma \ref{L:key2single} there exists a measure one set on which for every $z \in S^{1}$,
\[
\lim_{N \to \infty} \frac{1}{N}\sum_{n=1}^{N}z^{Y_{n}} - \mathbb{E}\frac{1}{N}\sum_{n=1}^{N}z^{Y_{n}} = 0.
\]
To conclude the proof, note that
\[
\mathbb{E}z^{Y_{n}}
= \mathbb{E}z^{a_{n+\ell}^{(k)} - a_{n}^{(k)}}
= \mathbb{E}z^{a_{\ell}^{(k)} - a_{0}^{(k)}} \circ \Theta^{n}
= \mathbb{E}z^{a_{\ell}^{(k)} - a_{0}^{(k)}}.
\]
\end{proof}

\begin{proposition}\label{P:weakpowerergodic}
Let $\{ a_{n} \}$ be a stochastic sequence generated by a random variable $b$ with finite mean.  Let $\ell$ and $q$ be fixed positive integers.  Set, for each $k \in \mathbb{N}$ with $k > \ell + q$,
\[
Y_{n}^{(k,\ell,q)} = a_{n+\ell+q}^{(k)} - a_{n+q}^{(k)} - a_{n+\ell}^{(k)} + a_{n}^{(k)}.
\]
Then there is a measure one set on which for every $z \in S^{1}$,
\[
\lim_{N \to \infty} \sup_{\ell+q < k \leq N} \Big{|}\frac{1}{N}\sum_{n=1}^{N}z^{Y_{n}^{(k,\ell,q)}} - \mathbb{E}z^{Y_{0}^{(k,\ell,q)}}\Big{|} = 0.
\]
\end{proposition}
\begin{proof}
Observe that
\[
a_{n+\ell+q}^{(k)} - a_{n+q}^{(k)} = a_{n+q+k}^{(\ell)} - a_{n+q}^{(\ell)}
\]
and so
\begin{align*}
a_{n+\ell+q}^{(k)} - a_{n+q}^{(k)} - a_{n+\ell}^{(k)} + a_{n}^{(k)}
&= a_{n+q+k}^{(\ell)} - a_{n+q}^{(\ell)} - a_{n+k}^{(\ell)} + a_{n}^{(\ell)} \\
&= a_{n+q+k}^{(\ell)} - a_{n+k}^{(\ell)} - a_{n+q}^{(\ell)} + a_{n}^{(\ell)} \\
&= (a_{q}^{(\ell)} - a_{0}^{(\ell)}) \circ \Theta^{n+k} - (a_{q}^{(\ell)} - a_{0}^{(\ell)}) \circ \Theta^{n}
\end{align*}
and therefore depends only on the coordinates $b_{n+1}, \ldots, b_{n+q+\ell}$ and $b_{n+k+1},\ldots,b_{n+k+q+\ell}$.
Therefore for each $n,k$,
\[
\# \{ m : Y_{n}^{(k,\ell,q)} \text{ depends on } b_{m} \} \leq 2(q + \ell)
\]
and for each $m,k$,
\[
\# \{ n : Y_{n}^{(k,\ell,q)} \text{ depends on } b_{m} \} \leq 2(q + \ell).
\]
Observe further that
\begin{align*}
|Y_{n}^{(k,\ell,q)}| &\leq (a_{q}^{(\ell)} - a_{0}^{(\ell)}) \circ \Theta^{n+k} + (a_{q}^{(\ell)} - a_{0}^{(\ell)}) \circ \Theta^{n} \\
&\leq a_{q}^{(\ell)} \circ \Theta^{n+k} + a_{q}^{(\ell)} \circ \Theta^{n} \\
&\leq \ell a_{q + \ell} \circ \Theta^{n+k} + \ell a_{q + \ell} \circ \Theta^{n}
\end{align*}
and that
\[
\sup_{k \leq N} \frac{1}{N}\sum_{n=1}^{N}b_{t} \circ \Theta^{n+k} \leq \frac{1}{N}\sum_{n=1}^{2N} b_{t} \circ \Theta^{n}
= 2 \frac{1}{2N}\sum_{n=1}^{2N} b_{t} \circ \Theta^{n}
\]
hence
\[
\limsup_{N\to\infty} \sup_{k \leq N} \frac{1}{N}\sum_{n=1}^{N}b_{t}\circ \Theta^{n+k} \leq 2 \mathbb{E}b
\]
almost surely by the Strong Law of Large Numbers.  Therefore
\begin{align*}
\limsup_{N\to\infty} \sup_{\ell + q < k \leq N} \frac{1}{N}\sum_{n=1}^{N} Y_{n}^{(k,\ell,q)}
&\leq \ell \sum_{t=1}^{q+\ell} \limsup_{N\to\infty} \sup_{\ell + q < k \leq N} \frac{1}{N}\sum_{n=1}^{N} b_{t}\circ\Theta^{n+k} + b_{t}\circ\Theta^{n} \\
&\leq \ell \sum_{t=1}^{q+\ell} 4 \mathbb{E}b = 4 \ell (q + \ell) \mathbb{E}b
\end{align*}
almost surely.  Hence by Lemma \ref{L:key2},
\[
\lim_{N \to \infty} \sup_{\ell + q < k \leq N} \Big{|}\frac{1}{N}\sum_{n=1}^{N}z^{Y_{n}^{(k,\ell,q)}} - \mathbb{E}\frac{1}{N}\sum_{n=1}^{N}z^{Y_{n}^{(k,\ell,q)}}\Big{|} = 0
\]
and the claim then follows from the fact that $Y_{n}^{(k,\ell,q)} = Y_{0}^{(k,\ell,q)} \circ \Theta^{n}$ by Lemma \ref{L:remarks}.
\end{proof}

\begin{remark}\label{R:blah}
The proofs of the above results carry over to the case when $a_{n} = b_{1} + \cdots + b_{n}$ is replaced by $a_{N,n} = b_{N,1} + \cdots + b_{N,n}$ where $b_{N,j}$ is an iid sequence for each $N$ with distribution $b_{N}$.  The only requirement for the above proofs is that $\mathbb{E}b_{N}$ be uniformly bounded over $N$.

However, if $\mathbb{E}b_{N}$ is bounded by a polynomial in $N$ then the statements remain true, provided we increase the $R$ used in Lemma \ref{L:key}.  In essence, we have a Strong Law for Triangular Arrays with a mild requirement on the means of the $b_{N}$.
\end{remark}

\subsection{The van der Corput Inequality}
A fundamental inequality in ergodic theory is the \textbf{van der Corput Inequality} (\cite{KN} page 25):
\begin{lemma*}
For any complex numbers $c_{n}$ such that $|c_{n}|\leq 1$ and any $N,L\in\mathbb{N}$,
\[
\big{|}\frac{1}{N}\sum_{n=0}^{N-1}c_{n}\big{|}^{2} \leq \frac{N+L}{N}\Big{(}\frac{1}{L} + 2Re\Big{[}\frac{1}{L}\sum_{\ell=1}^{L-1}\frac{L-\ell}{L}\frac{1}{N}\sum_{n=0}^{N-\ell-1}c_{n+\ell}\overline{c_{n}}\Big{]}\Big{)}.
\]
\end{lemma*}
The van der Corput Inequality is now a classical tool in ergodic theory and is used, among other places, in \cite{adba00}, \cite{bergelson87}, \cite{fw96}, \cite{hostkra} and \cite{lesigne93}.

We need several consequences of this basic inequality.  The first is straightforward and the proof is left to the reader:
\begin{lemma}\label{L:vdc2}
For complex numbers $c_{n}$ with $|c_{n}| \leq 1$ and any $N,L \in \mathbb{N}$,
\[
\big{|}\frac{1}{N}\sum_{n=1}^{N}c_{n}\big{|}^{2} \leq \frac{N+L}{N} \Big{(}\frac{1}{L} + 2 Re \Big{[}\frac{1}{L}\sum_{\ell=1}^{L-1} \frac{L-\ell}{L}\frac{1}{N}\sum_{n=1}^{N}c_{n+\ell}\overline{c_{n}} \Big{]}\Big{)} + \frac{2L(N+L)}{N^{2}}.
\]
\end{lemma}
%\begin{proof}
%Beginning with the van der Corput Inequality,
%\begin{align*}
%\big{|}\frac{1}{N}\sum_{n=1}^{N}c_{n}\big{|}^{2}
%&\leq \frac{N+L}{N}\Big{(}\frac{1}{L} + 2Re\Big{[}\frac{1}{L}\sum_{\ell=1}^{L-1}\frac{L-\ell}{L}\frac{1}{N}\sum_{n=0}^{N-\ell-1}c_{n+\ell}\overline{c_{n}}\Big{]}\Big{)} \\
%&= \frac{N+L}{N} \Big{(} \frac{1}{L} + 2 Re \Big{[}\frac{1}{L}\sum_{\ell=1}^{L-1} \frac{L-\ell}{L}\frac{1}{N}\sum_{n=1}^{N}c_{n+\ell}\overline{c_{n}} \Big{]}\Big{)}  \\
%&\quad\quad\quad\quad\quad\quad - \frac{N+L}{N} 2 Re \Big{[}\frac{1}{L}\sum_{\ell=1}^{L-1}\frac{L-\ell}{L} \frac{1}{N}\sum_{n=N-\ell+1}^{N} c_{n+\ell}\overline{c_{n}} \Big{]} \\
%&\leq \frac{N+L}{N} \Big{(} \frac{1}{L} + 2 Re \Big{[} \frac{1}{L}\sum_{\ell=1}^{L-1} \frac{L-\ell}{L}\frac{1}{N}\sum_{n=1}^{N}c_{n+\ell}\overline{c_{n}} \Big{]} \Big{)} 
%+ \frac{N+L}{N} 2 \frac{1}{L}\sum_{\ell=1}^{L-1} \frac{L-\ell}{L}\frac{L}{N}.
%\end{align*}
%\end{proof}

\begin{lemma}\label{L:vdc3}
For any sequence $\{a_{n}\}$, any $z \in S^{1}$ and any fixed $k,L \in \mathbb{N}$,
\[
\limsup_{N\to\infty} \big{|}\frac{1}{N}\sum_{n=1}^{N}z^{a_{n}^{(k)}}\big{|}^{2}
\leq \frac{1}{L} + 2\frac{1}{L}\sum_{\ell=1}^{L-1}\frac{L-\ell}{L}\limsup_{N\to\infty}Re\Big{[}\frac{1}{N}\sum_{n=1}^{N}z^{a_{n+\ell}^{(k)} - a_{n}^{(k)}}\Big{]}.
\]
\end{lemma}
\begin{proof}
Apply Lemma \ref{L:vdc2} to $c_{n} = z^{a_{n}^{(k)}}$.
\end{proof}

Our next consequence of the van der Corput Inequality is similar to the fourth moment method of Blum and Cogburn \cite{blumcogburn}.
\begin{lemma}\label{L:vdc4}
For any sequence $\{a_{n}\}$, any $z \in S^{1}$ and any $L,Q \in \mathbb{N}$,
\begin{align*}
&\limsup_{N\to\infty}\sup_{k\leq N}\big{|}\frac{1}{N}\sum_{n=1}^{N}z^{a_{n}^{(k)}}\big{|}^{4} \\
&\quad\leq \frac{1}{L^{2}} + \frac{4}{L} + \frac{4}{L}\sum_{\ell = 1}^{L} \Big{(}\frac{1}{Q} + 2\frac{1}{Q}\sum_{q=1}^{Q-1}\frac{Q-q}{Q}\limsup_{N\to\infty}\sup_{k\leq N} Re\Big{[}\frac{1}{N}\sum_{n=1}^{N}z^{a_{n+q+\ell}^{(k)} - a_{n+q}^{(k)} - a_{n+\ell}^{(k)} + a_{n}^{(k)}} \Big{]}\Big{)}.
\end{align*}
\end{lemma}

\begin{proof}
Using Lemma \ref{L:vdc2}, for any sequence of complex numbers $c_{n}$ with $|c_{n}| \leq 1$
\begin{align*}
\big{|}\frac{1}{N}\sum_{n=1}^{N}c_{n}\big{|}^{2}
&\leq \frac{N+L}{N}\Big{(}\frac{1}{L} + 2 Re \Big{[}\frac{1}{L}\sum_{\ell=1}^{L-1}\frac{L-\ell}{L}\frac{1}{N}\sum_{n=1}^{N} c_{n+\ell}\overline{c_{n}}\Big{]}\Big{)} + \frac{2L(N+L)}{N^{2}} \\
&\leq \frac{N+L}{N} \Big{(}\frac{1}{L} + 2 \big{|}\frac{1}{L}\sum_{\ell=1}^{L-1}\frac{L-\ell}{L}\frac{1}{N}\sum_{n=1}^{N}c_{n+\ell}\overline{c_{n}}\big{|}\Big{)} + \frac{2L(N+L)}{N^{2}} 
\end{align*}
and therefore (since $\big{|}\frac{1}{L}\sum_{\ell=1}^{L-1}\frac{L-\ell}{L}\frac{1}{N}\sum_{n=1}^{N}c_{n+\ell}\overline{c_{n}}\big{|} \leq \frac{1}{L}\sum_{\ell=1}^{L-1}\frac{L-\ell}{L}\frac{1}{N}\sum_{n=1}^{N} 1 \leq 1$)
\begin{align*}
\big{|}\frac{1}{N}&\sum_{n=1}^{N}c_{n}\big{|}^{4} \\
&\leq \frac{(N+L)^{2}}{N^{2}}\Big{(}\frac{1}{L^{2}} + \frac{4}{L} + 4 \big{|}\frac{1}{L}\sum_{\ell=1}^{L-1}\frac{L-\ell}{L}\frac{1}{N}\sum_{n=1}^{N}c_{n+\ell}\overline{c_{n}}\big{|}^{2}\Big{)} \\
&\quad\quad\quad\quad + \frac{12L(N+L)^{2}}{N^{3}} + \frac{4L^{2}(N+L)^{2}}{N^{4}} \\
&\leq \frac{(N+L)^{2}}{N^{2}}\Big{(}\frac{1}{L^{2}} + \frac{4}{L} + \frac{4}{L}\sum_{\ell=1}^{L}\big{(}\frac{L - \ell}{L}\big{)}^{2}\big{|}\frac{1}{N}\sum_{n=1}^{N}c_{n+\ell}\overline{c_{n}}\big{|}^{2}\Big{)} \\
&\quad\quad\quad\quad + \frac{12L(N+L)^{2}}{N^{3}} + \frac{4L^{2}(N+L)^{2}}{N^{4}} \\
&\leq \frac{(N+L)^{2}}{N^{2}}\Big{(}\frac{1}{L^{2}} + \frac{4}{L} + \frac{4}{L}\sum_{\ell=1}^{L}\big{|}\frac{1}{N}\sum_{n=1}^{N}c_{n+\ell}\overline{c_{n}}\big{|}^{2}\Big{)} + \frac{12L(N+L)^{2}}{N^{3}} + \frac{4L^{2}(N+L)^{2}}{N^{4}}
\end{align*}
where the second inequality is the Cauchy-Schwarz inequality applied as follows:
\[
\Big{|}\frac{1}{L}\sum_{\ell=1}^{L} a_{\ell}\Big{|}^{2} = \Big{|}\sum_{\ell=1}^{L} \frac{a_{\ell}}{L} \Big{|}^{2} \leq \sum_{\ell=1}^{L}\Big{|}\frac{a_{\ell}}{L}\Big{|}^{2} \sum_{\ell=1}^{L} 1^{2} = \frac{1}{L^{2}}\sum_{\ell=1}^{L}\big{|}a_{\ell}\big{|}^{2} L = \frac{1}{L}\sum_{\ell=1}^{L}\big{|}a_{\ell}\big{|}^{2}
\]
Therefore
\[
\limsup_{N\to\infty} \sup_{k \leq N} \big{|}\frac{1}{N}\sum_{n=1}^{N}z^{a_{n}^{(k)}}\big{|}^{4}
\leq \frac{1}{L^{2}} + \frac{4}{L} + \frac{4}{L}\sum_{\ell = 1}^{L}\limsup_{N\to\infty} \sup_{k \leq N} \big{|}\frac{1}{N}\sum_{n=1}^{N}z^{a_{n+\ell}^{(k)} - a_{n}^{(k)}}\big{|}^{2}
\]
since the supremum of an average is bounded by the average of the supremums (and that limits of finite sums interchange).  Applying Lemma \ref{L:vdc2} a second time,
\begin{align*}
\limsup_{N\to\infty} \sup_{k \leq N} &\big{|}\frac{1}{N}\sum_{n=1}^{N}z^{a_{n+\ell}^{(k)} - a_{n}^{(k)}}\big{|}^{2} \\
&\leq \frac{1}{Q} + 2 \frac{1}{Q}\sum_{q=1}^{Q-1}\frac{Q-q}{Q} \limsup_{N\to\infty}\sup_{k \leq N} Re\Big{[} \frac{1}{N}\sum_{n=1}^{N} z^{a_{n+q+\ell}^{(k)} - a_{n+\ell}^{(k)} - a_{n+q}^{(k)} + a_{n}^{(k)}}\Big{]}.
\end{align*}
\end{proof}

\subsection{Ergodicity Properties of Stochastic Sequences}

\begin{theorem}\label{T:ergodic}
Let $b$ be an aperiodic integer-valued random variable with finite mean.  
Then almost every sequence stochastically generated by $b$ is ergodic with respect to any ergodic transformation.
\end{theorem}

\begin{remark*}
Theorem \ref{T:ergodic} is a special case of the general result of Lema\'{n}czyk, Lesigne, Parreau, Voln\'{y} and Wierdl \cite{leman} on the ergodicity of sequences obtained from measure-preserving transformations.  We include a brief proof to illustrate the techniques that will be used in the proof of Theorem \ref{T:P} (which does not follow from their work).

Specifically, they show that if $\tau$ is an ergodic measure-preserving transformation and $F$ is a measurable integer-valued function then for a.e.~$\omega$, the sequence $k_{n}(\omega) = \sum_{j=1}^{n} F(\tau^{j}(\omega))$ is ergodic with respect to every ergodic transformation provided that $F$ satisfies a cohomology condition (which holds when $F$ is taken to be our function $b$).
\end{remark*}

\begin{proof}
The measure one set will be the intersection of the measure one sets from Proposition \ref{P:ergodic} for each fixed $\ell$ (countably many $\ell$).  By Proposition \ref{P:ergodic} for every $z \in S^{1}$ and each fixed $\ell \in \mathbb{N}$,
$\lim_{N \to \infty} \frac{1}{N}\sum_{n=1}^{N}z^{a_{n+\ell} - a_{n}} = \mathbb{E}z^{a_{\ell}}$.
Observe that
$\mathbb{E}z^{a_{\ell}} = \mathbb{E}z^{b_{1} + \cdots + b_{\ell}} = \big{(}\mathbb{E}z^{b}\big{)}^{\ell}$.
Now by Lemma \ref{L:vdc3} (with $k = 1$) we know that for every $z \in S^{1}$ and $L \in \mathbb{N}$,
\[
\limsup_{N\to\infty} \big{|}\frac{1}{N}\sum_{n=1}^{N}z^{a_{n}}\big{|}^{2}
\leq \frac{1}{L} + 2\frac{1}{L}\sum_{\ell=1}^{L-1}\frac{L-\ell}{L}\limsup_{N\to\infty}Re\Big{[}\frac{1}{N}\sum_{n=1}^{N}z^{a_{n+\ell} - a_{n}}\Big{]}
\]
and this means that for all $L \in \mathbb{N}$,
\begin{align*}
\limsup_{N\to\infty} \big{|}\frac{1}{N}\sum_{n=1}^{N}z^{a_{n}}\big{|}^{2}
&\leq \frac{1}{L} + 2\frac{1}{L}\sum_{\ell=1}^{L-1}\frac{L-\ell}{L} Re \big{(}\mathbb{E}z^{b}\big{)}^{\ell}
= \big{|}\frac{1}{L}\sum_{\ell=1}^{L}  \big{(}\mathbb{E}z^{b}\big{)}^{\ell} \big{|}^{2}.
\end{align*}

 Since $|z^{b}| \leq 1$ we have that $\mathbb{E}z^{b} = 1$ with equality if and only if $z^{b} = 1$ almost surely.  But $b$ is aperiodic so this can happen if and only if $z = 1$ (irrational $z$ this cannot happen and rational $z$ would require $b$ be periodic).  Therefore, for $z \ne 1$ (taking $L \to \infty$),
$\lim_{N \to \infty} \frac{1}{N}\sum_{n=1}^{N}z^{a_{n}} = 0$.
Let $\sigma$ be a spectral measure for an ergodic transformation so $\sigma(\{ 1 \}) = 0$.  By Dominated Convergence,
$\lim_{N \to \infty} \int \big{|}\frac{1}{N}\sum_{n=1}^{N}z^{a_{n}}\big{|}^{2} d\sigma(z) = \sigma(\{ 1 \}) = 0$.
\end{proof}

\begin{theorem}\label{T:totallyergodic}
Let $b$ be an aperiodic integer-valued random variable with finite mean.  
Then almost every sequence stochastically generated by $b$ is totally ergodic with respect to any totally ergodic transformation.
\end{theorem}

\begin{remark*}
Theorem \ref{T:totallyergodic} is also a consequence of the result of Lema\'{n}czyk, Lesigne, Parreau, Voln\'{y} and Wierdl \cite{leman} but we include the proof since it is actually also part of the proof of Theorem \ref{T:P} (which does not follow from their work).
\end{remark*}

\begin{proof}
Proceeding as in the previous theorem, the measure one set will be that from Proposition \ref{P:totallyergodic} intersected over all $\ell$ and on that set we have for each $z \in S^{1}$ and $k,\ell \in \mathbb{N}$,
\begin{align*}
\lim_{N \to \infty} \frac{1}{N}\sum_{n=1}^{N}z^{a_{n+\ell}^{(k)} - a_{n}^{(k)}} = \mathbb{E}z^{a_{\ell}^{(k)} - a_{0}^{(k)}}.
\end{align*}
Now by Lemma \ref{L:vdc3} we have that
\begin{align*}
\limsup_{N\to\infty} \big{|}\frac{1}{N}\sum_{n=1}^{N}z^{a_{n}^{(k)}}\big{|}^{2}
&\leq \frac{1}{L} + 2\frac{1}{L}\sum_{\ell=1}^{L-1}\frac{L-\ell}{L}\limsup_{N\to\infty}Re\Big{[}\frac{1}{N}\sum_{n=1}^{N}z^{a_{n+\ell}^{(k)} - a_{n}^{(k)}}\Big{]} \\
&= \frac{1}{L} + 2\frac{1}{L}\sum_{\ell=1}^{L-1}\frac{L-\ell}{L}Re \mathbb{E}z^{a_{\ell}^{(k)} - a_{0}^{(k)}}.
\end{align*}
Observe that for $\ell > k$
\begin{align*}
a_{\ell}^{(k)} - a_{0}^{(k)} = k a_{\ell} + a_{0}^{(k)} \circ \Theta^{\ell} - a_{0}^{(k)}
&= a_{0}^{(k)} \circ \Theta^{\ell} + k (a_{\ell} - a_{k}) + ka_{k} - a_{0}^{(k)} \\
&= a_{0}^{(k)} \circ \Theta^{\ell} + k a_{\ell - k} \circ \Theta^{k} + (k a_{k} - a_{0}^{(k)})
\end{align*}
and each of the three terms above is independent so for $\ell > k$,
\[
\mathbb{E}z^{a_{\ell}^{(k)} - a_{0}^{(k)}} = \mathbb{E}z^{a_{0}^{(k)}} \mathbb{E}z^{k a_{\ell - k}} \mathbb{E}z^{k a_{k} - a_{0}^{(k)}}
= \big{(}\mathbb{E}z^{kb}\big{)}^{\ell - k} \big{(}\mathbb{E}z^{a_{0}^{(k)}}\big{)}^{2}
\]
since $a_{0}^{(k)} = b_{1} + 2b_{2} + \cdots + (k - 1)b_{k-1}$ and $k a_{k} - a_{0}^{(k)} = b_{k-1} + 2b_{k-2} + \cdots + (k-1)b_{1}$ hence they have the same expectation.
Therefore for $L > k$,
\begin{align*}
\frac{1}{L} &+ 2\frac{1}{L}\sum_{\ell=1}^{L-1}\frac{L-\ell}{L}Re \big{[}\mathbb{E}z^{a_{\ell}^{(k)} - a_{0}^{(k)}}\big{]} \\
&= \frac{1}{L} + 2 \frac{1}{L}\sum_{\ell=k+1}^{L-1}\frac{L-\ell}{L}Re\big{[}\mathbb{E}z^{a_{\ell}^{(k)} - a_{0}^{(k)}}\big{]} + 2 \frac{1}{L}\sum_{\ell=1}^{k}\frac{L-\ell}{L}Re\mathbb{E}z^{a_{\ell}^{(k)} - a_{0}^{(k)}} \\
&\leq \frac{1}{L} + 2\frac{1}{L}\sum_{\ell=k+1}^{L-1}\frac{L-\ell}{L}Re  \Big{[} \big{(}\mathbb{E}z^{kb}\big{)}^{\ell - k} \big{(}\mathbb{E}z^{a_{0}^{(k)}}\big{)}^{2}\Big{]} + 2 \frac{1}{L}\sum_{\ell=1}^{k}\frac{L-\ell}{L}1 \\
&= \frac{1}{L} + 2 Re \Big{[}\frac{1}{L}\sum_{\ell=1}^{L-k-1}\frac{L-\ell -k}{L}\big{(}\mathbb{E}z^{kb}\big{)}^{\ell} \big{(}\mathbb{E}z^{a_{0}^{(k)}}\big{)}^{2}\Big{]} + 2 \frac{1}{L}\sum_{\ell=1}^{k}\frac{L-\ell}{L} \\
&\leq \frac{1}{L} + 2 Re \Big{[}\frac{1}{L} \sum_{\ell=1}^{L-k-1}\frac{L-\ell}{L}\big{(}\mathbb{E}z^{kb}\big{)}^{\ell} \big{(}\mathbb{E}z^{a_{0}^{(k)}}\big{)}^{2}\Big{]} - 2 Re \Big{[}\frac{1}{L}\sum_{\ell=1}^{L-k-1}\frac{k}{L}\big{(}\mathbb{E}z^{kb}\big{)}^{\ell} \big{(}\mathbb{E}z^{a_{0}^{(k)}}\big{)}^{2}\Big{]} + 2 \frac{k}{L} \\
&\leq \frac{1}{L} + 2 Re \Big{[} \frac{1}{L} \sum_{\ell=1}^{L-1}\frac{L-\ell}{L}\big{(}\mathbb{E}z^{kb}\big{)}^{\ell} \big{(}\mathbb{E}z^{a_{0}^{(k)}}\big{)}^{2} \Big{]} + \frac{6k}{L} \\
&\leq \frac{1}{L} + 2 \frac{1}{L}\sum_{\ell=1}^{L-1}\frac{L-\ell}{L}\big{|}\mathbb{E}z^{kb}\big{|}^{\ell}\mathbb{E}|z^{a_{0}^{(k)}}|^{2} + \frac{6k}{L} \\
&\leq \frac{2}{L}\sum_{\ell=0}^{L-1}\big{|}\mathbb{E}z^{kb}\big{|}^{\ell} + \frac{6k}{L}.
\end{align*}
Therefore, for every $L \in \mathbb{N}$ we have that
\[
\limsup_{N\to\infty} \big{|}\frac{1}{N}\sum_{n=1}^{N}z^{a_{n}^{(k)}}\big{|}^{2} \leq \frac{1}{L}\sum_{\ell=1}^{L}\big{|}\mathbb{E}z^{kb}\big{|}^{\ell} + \frac{6k}{L}.
\]
Since $k$ is fixed, taking $L \to \infty$ sends the final term to zero.  As in the previous theorem, $| \mathbb{E}z^{kb} | = 1$ if and only if $z$ is a root of unity (in fact a $kp^{th}$ root where $p$ is the period of $b$).  Hence, taking $L \to \infty$ we obtain that for $z \in S^{1}$, $z$ not a root of unity,
$\lim_{N \to \infty} \frac{1}{N}\sum_{n=1}^{N}z^{a_{n}^{(k)}} = 0$.

Now let $\sigma$ be a spectral measure for a totally ergodic transformation.  Then, since if there were some mass on a $t^{th}$ root of unity then $T^{t}$ would not be ergodic, 
$\sigma(\{ \text{roots of unity} \}) = 0$.  Hence by Dominated Convergence,
\[
\limsup_{N\to\infty} \int \big{|}\frac{1}{N}\sum_{n=1}^{N}z^{a_{n}^{(k)}}\big{|}^{2} d\sigma(z) \leq \sigma(\{ \text{roots of unity} \}) = 0.
\]
\end{proof}

\begin{theorem}\label{T:P}
Let $b$ be an aperiodic integer-valued random variable with finite mean.  
Then almost every sequence stochastically generated by $b$ is weakly power ergodic with respect to any weakly power ergodic transformation.
\end{theorem}
\begin{proof}
The measure one set will be the intersection of the measure one sets provided by Proposition \ref{P:weakpowerergodic} for each pair $q,\ell \in \mathbb{N}$ (countably many measure one sets).
Assume for the moment that $b$ is not constant.
Set, for each $k \in \mathbb{N}$ with $k > \ell + q$,
\[
Y_{n}^{(k,\ell,q)} = a_{n+\ell+q}^{(k)} - a_{n+q}^{(k)} - a_{n+\ell}^{(k)} + a_{n}^{(k)}.
\]
On the measure one set chosen, by Proposition \ref{P:weakpowerergodic}, for every $\ell,q$ and every $z \in S^{1}$,
\[
\lim_{N \to \infty} \sup_{\ell + q < k \leq N} \Big{|}\frac{1}{N}\sum_{n=1}^{N}z^{Y_{n}^{(k,\ell,q)}} - \mathbb{E}z^{Y_{0}^{(k,\ell,q)}}\Big{|} = 0
\]
and for $k \leq \ell + q$ the proof of the previous theorem has already established this, hence we may take the supremum of $k \leq N$.
Now provided $\ell < q$
\begin{align*}
Y_{0}^{(k,\ell,q)} &= a_{\ell+q}^{(k)} - a_{q}^{(k)} - a_{\ell}^{(k)} + a_{0}^{(k)} \\
&= (a_{q}^{(\ell)} - a_{0}^{(\ell)}) \circ \Theta^{k} - (a_{q}^{(\ell)} - a_{0}^{(\ell)}) \\
&= \big{(}a_{0}^{(\ell)} \circ \Theta^{q} + \ell a_{q - \ell} \circ \Theta^{\ell} + (\ell a_{\ell} - a_{0}^{(\ell)})\big{)} \circ \Theta^{k}
- \big{(}a_{0}^{(\ell)} \circ \Theta^{q} + \ell a_{q - \ell} \circ \Theta^{\ell} + (\ell a_{\ell} - a_{0}^{(\ell)})\big{)} 
\end{align*}
which are six independent terms and therefore
\[
\mathbb{E}z^{Y_{0}^{(k,\ell,q)}} = \big{|}\mathbb{E}z^{\ell b}\big{|}^{2(q - \ell)} \big{|}\mathbb{E}z^{a_{0}^{(\ell)}}\big{|}^{4}
\]
as in the previous theorem.  Hence for $\ell < q$ we have that
\[
\limsup_{N \to \infty} \sup_{k \leq N} \big{|}\frac{1}{N}\sum_{n=1}^{N}z^{Y_{n}^{(k,\ell,q)}}\big{|}
\leq \big{|}\mathbb{E}z^{\ell b}\big{|}^{2(q - \ell)} \big{|}\mathbb{E}z^{a_{0}^{(\ell)}}\big{|}^{4}
\leq \big{|}\mathbb{E}z^{\ell b}\big{|}^{2(q - \ell)}.
\]
By Lemma \ref{L:vdc4} we have that, for $Q > L$,
\begin{align*}
&\limsup_{N\to\infty}\sup_{k\leq N}\big{|}\frac{1}{N}\sum_{n=1}^{N}z^{a_{n}^{(k)}}\big{|}^{4} \\
&\quad\leq \frac{1}{L^{2}} + \frac{4}{L} + \frac{4}{L}\sum_{\ell = 1}^{L} \Big{(}\frac{1}{Q} + 2\frac{1}{Q}\sum_{q=1}^{Q-1}\frac{Q-q}{Q}\limsup_{N\to\infty}\sup_{k\leq N} Re\Big{[}\frac{1}{N}\sum_{n=1}^{N}z^{a_{n+q+\ell}^{(k)} - a_{n+q}^{(k)} - a_{n+\ell}^{(k)} + a_{n}^{(k)}} \Big{]}\Big{)} \\
&\quad\leq \frac{1}{L^{2}} + \frac{4}{L} + \frac{4}{L}\sum_{\ell = 1}^{L} \Big{[}\frac{1}{Q} + 2\frac{1}{Q}\sum_{q=\ell+1}^{Q-1}\frac{Q-q}{Q} \big{|}\mathbb{E}z^{\ell b}\big{|}^{2(q - \ell)} + 2\frac{1}{Q}\sum_{q=1}^{\ell}\frac{Q-q}{Q}1 \Big{]} \\
&\quad\leq \frac{1}{L^{2}} + \frac{4}{L} + \frac{4}{L}\sum_{\ell=1}^{L}\Big{[}\frac{1}{Q} + 2\frac{1}{Q}\sum_{q=1}^{Q-\ell-1}\frac{Q-q-\ell}{Q}\big{|}\mathbb{E}z^{\ell b}\big{|}^{2q} + 2 \frac{\ell}{Q}\Big{]} \\
&\quad\leq  \frac{1}{L^{2}} + \frac{4}{L} + \frac{4}{L}\sum_{\ell=1}^{L}\Big{[}\frac{1}{Q} + 2\frac{1}{Q}\sum_{q=1}^{Q-1}\frac{Q-q}{Q}\big{|}\mathbb{E}z^{\ell b}\big{|}^{2q} \Big{]} + 2 \frac{4}{L}\sum_{\ell=1}^{L} \frac{\ell}{Q} \\
&\quad= \frac{1}{L^{2}} + \frac{4}{L} + \frac{4}{L}\sum_{\ell=1}^{L} \Big{|}\frac{1}{Q}\sum_{q=1}^{Q} \big{|}\mathbb{E}z^{\ell b}\big{|}^{2q} \Big{|}^{2} + \frac{8}{L}\frac{L(L+1)}{2Q}.
\end{align*}
Fix $\epsilon > 0$ and choose $L$ such that $\frac{1}{L^{2}} + \frac{4}{L} < \epsilon$.  Choose $Q$ at least large enough that $\frac{4(L+1)}{Q} < \epsilon$.  Then
\[
\limsup_{N\to\infty}\sup_{k\leq N}\big{|}\frac{1}{N}\sum_{n=1}^{N}z^{a_{n}^{(k)}}\big{|}^{4}
\leq \frac{4}{L}\sum_{\ell=1}^{L} \Big{|}\frac{1}{Q}\sum_{q=1}^{Q} \big{|}\mathbb{E}z^{\ell b}\big{|}^{2q} \Big{|}^{2} + 2\epsilon.
\]
Now, as before, $|\mathbb{E}z^{\ell b}| = 1$ can only occur when $z$ is a root of unity.  Hence, when $z$ is not a root of unity, for each of the finite number of choices for $\ell \leq L$ there is large enough $Q$ such that $\frac{1}{Q}\sum_{q=1}^{Q} \big{|}\mathbb{E}z^{\ell b}\big{|}^{2q} < \epsilon$.  Therefore for $z \in S^{1}$, $z$ not a root of unity, we have that
\[
\limsup_{N\to\infty}\sup_{k\leq N}\big{|}\frac{1}{N}\sum_{n=1}^{N}z^{a_{n}^{(k)}}\big{|}^{4}
\leq \frac{4}{L}\sum_{\ell=1}^{L} \epsilon^{2} + 2\epsilon
\]
and therefore for $z$ not a root of unity (since $\epsilon$ was arbitrary),
\[
\lim_{N\to\infty}\sup_{k\leq N}\big{|}\frac{1}{N}\sum_{n=1}^{N}z^{a_{n}^{(k)}}\big{|}^{2} = 0.
\]
Hence by Dominated Convergence, for any spectral measure $\sigma$ for a weakly power ergodic transformation,
\[
\lim_{N \to \infty} \sup_{k \leq N} \int \big{|}\frac{1}{N}\sum_{n=1}^{N}z^{a_{n}^{(k)}}\big{|}^{2} d\sigma(z) = 0
\]
as desired.

The case when $b$ is constant corresponds to $a_{n} = n$ and therefore $a_{n}^{(k)} = n + (n+1) + \cdots + n + k -1 = kn + \frac{1}{2}k(k-1)$.  Hence for $\sigma$ a weakly power ergodic spectral measure,
\[
\lim_{N \to \infty} \sup_{k \leq N} \int \big{|}\frac{1}{N}\sum_{n=1}^{N}z^{a_{n}^{(k)}}\big{|}^{2} d\sigma(z)
= \lim_{N \to \infty} \sup_{k \leq N} \int \big{|}\frac{1}{N}\sum_{n=1}^{N}z^{nk}\big{|}^{2} d\sigma(z) = 0
\]
by the definition of weak power ergodicity.
\end{proof}

\begin{remark}\label{R:blah2}
The proofs carry over to stochastic dynamical sequences $a_{N,n}$ generated by $\{ b_{N} \}$ provided $\mathbb{E}b_{N}$ are bounded by some polynomial in $N$ as per Remark \ref{R:blah}.  Stochastic dynamical sequences satisfying this condition are likewise totally (respectively, weakly) power ergodic with respect to totally (respectively, weakly) power ergodic transformations
\end{remark}

\section{Mixing on Stochastic Staircase Transformations}

\begin{theorem}\label{T:rsmix}
Let $b$ be an aperiodic positive-integer-valued random variable with finite mean and $\{ r_{n} \}$ be a sequence of positive integers with $r_{n} \to \infty$ such that almost every stochastic staircase transformation generated by $b$ with cut sequence $\{ r_{n} \}$ is defined on a finite measure space.  Then almost every stochastic staircase transformation generated by $b$ with cut sequence $\{ r_{n} \}$ is mixing.
\end{theorem}
\begin{proof}
Let $\{ a_{n} \}$ be a stochastic sequence generated by $b$ such that $\{ a_{n} \}$ is in the measure one set of such sequences for Theorems \ref{T:ergodic}, \ref{T:totallyergodic} and \ref{T:P}.  Let $T$ be the stochastic staircase transformation with spacer sequence $\{ s_{n,j} \}_{\{r_{n}\}}$ where $s_{n,j} = a_{j} + x_{n}$ where $\{ x_{n} \}$ is a sequence of nonnegative integers such that $x_{n} \geq -\inf_{1\leq j< r_{n}}a_{j}$.

Since $T$ is a rank-one transformation, $T$ is ergodic.  Then by Theorem \ref{T:ergodic}, $\{a_{n}\}$ is ergodic with respect to $T$.  Therefore, for any $\chi \in L^{2}(X,\mu)$ with $\int \chi~d\mu = 0$,
\begin{align*}
\int \big{|} \frac{1}{r_{n}}\sum_{j=0}^{r_{n}-1}\chi\circ T^{-s_{n,j}} \big{|}~d\mu
&= \int \big{|} \frac{1}{r_{n}}\sum_{j=0}^{r_{n}-1}\chi\circ T^{-a_{j}-x_{n}} \big{|}~d\mu \\
&= \int \big{|} \frac{1}{r_{n}}\sum_{j=0}^{r_{n}-1}\chi\circ T^{-a_{j}} \big{|}~d\mu \to 0
\end{align*}
using that $T$ is measure-preserving.
Hence by Corollary \ref{C:r1te}, $T$ is totally ergodic and so by Theorem \ref{T:totallyergodic}, $\{a_{n}\}$ is then totally ergodic with respect to $T$.  Continuing this process, by Theorem \ref{T:r1pe}, $T$ is weakly power ergodic and then Theorem \ref{T:P} tells us that $\{ a_{n} \}$ is weakly power ergodic with respect to $T$.

By Theorem \ref{T:rgmix} this means that if $T$ has restricted growth then $T$ is mixing.  The case when $T$ does not have restricted growth will occupy the rest of the proof.  

First we rewrite the condition from Theorem \ref{T:mix1} that we need to show as:
\[
\lim_{n\to\infty} \sup_{k < r_{n}} 
\sup \big{\{}
 \int \sum_{q=1}^{Q}\frac{d_{q}}{r_{n}-k} \big{|}\frac{1}{d_{q}}\sum_{j=1}^{d_{q}}
z^{s_{n,j + \sum_{i=1}^{q-1}d_{i}}^{(k - \alpha_{q})}} \big{|} d\sigma(z)
: \sum_{q=1}^{Q} d_{q} = r_{n}-k \text{ and } \alpha_{q} \geq 0 \big{\}} = 0
\]
where we are assuming that $s_{n,j} \leq s_{n,j+1}$ so the partitions $\Gamma_{q}$ are just blocks of length $d_{q}$ (this is because the $s_{n,j}$ are increasing with $j$ and the blocks collect nearby values of $s_{n,j}$ together).  The case when the spacers are not increasing can be handled similarly (with additional notational difficulty).

By Lemma \ref{L:weird} (following the proof), it is enough to show that
\[
\lim_{n\to\infty} \sup_{k < r_{n}} 
\sup \big{\{}
 \int \sum_{q=1}^{Q}\frac{d_{q}}{r_{n}-k} \big{|}\frac{1}{d_{q}}\sum_{j=1}^{d_{q}}
z^{s_{n,j + \sum_{i=1}^{q-1}d_{i}}^{(k - \alpha_{q})}} \big{|}^{2} d\sigma(z)
: \sum_{q=1}^{Q} d_{q} = r_{n}-k \text{ and } \alpha_{q} \geq 0 \big{\}} = 0
\]
and so applying the van der Corput Inequality (Lemma \ref{L:vdc2}) it is enough to show that
\begin{align*}
\inf_{L} \lim_{n\to\infty} \sup_{k < r_{n}} 
\sup_{d_{q},\alpha_{q}}
& \int \sum_{q=1}^{Q}\frac{d_{q}}{r_{n}-k} \Big{(}\frac{1}{L} + \\
 &2Re\Big{[}\frac{1}{L}\sum_{\ell=1}^{L-1}\frac{L-\ell}{L} \frac{1}{d_{q}}\sum_{j=1}^{d_{q}}
z^{s_{n,j + \ell + \sum_{i=1}^{q-1}d_{i}}^{(k - \alpha_{q})} - s_{n,j + \sum_{i=1}^{q-1}d_{i}}^{(k - \alpha_{q})}}\Big{]} \Big{)} d\sigma(z) = 0.
\end{align*}
Now in the case when $b$ is constant, we know that
\[
s_{n,j + \ell + \sum_{i=1}^{q-1}d_{i}}^{(k - \alpha_{q})} - s_{n,j + \sum_{i=1}^{q-1}d_{i}}^{(k - \alpha_{q})} = \ell (k - \alpha_{q})
\]
and this condition becomes
\[
\inf_{L} \lim_{n\to\infty} \sup_{k < r_{n}} 
\sup_{d_{q},\alpha_{q}}
 \int \sum_{q=1}^{Q}\frac{d_{q}}{r_{n}-k} \Big{[}\frac{1}{L} + 2Re\frac{1}{L}\sum_{\ell=1}^{L-1}\frac{L-\ell}{L} \frac{1}{d_{q}}\sum_{j=1}^{d_{q}}
z^{\ell(k - \alpha_{q})} \Big{]} d\sigma(z) = 0.
\]
In \cite{CS10} it is shown that classical staircase transformations are mixing and the reader is referred to Theorem 4.1 in \cite{CS10} for details.  We assume from here on that $b$ is not constant.
Following the same strategy as in the proof of Lemma \ref{L:vdc4}, we can apply the van der Corput Inequality again and it becomes enough to show that
\[
\inf_{L,M} \lim_{n\to\infty} \sup_{k < r_{n}} 
\sup_{d_{q},\alpha_{q}}
 \int \sum_{q=1}^{Q}\frac{d_{q}}{r_{n}-k} \Big{[}\frac{1}{L^{2}} +\frac{4}{L} + \frac{4}{L}\sum_{\ell=1}^{L}\big{|}\frac{1}{M}\sum_{m=1}^{M-1}\frac{M-m}{M} \frac{1}{d_{q}}\sum_{j=1}^{d_{q}}
z^{Y_{j}^{(k,q,\ell,m)}} \big{|} \Big{]} d\sigma(z) = 0
\]
where
\[
Y_{j}^{(k,q,\ell,m)} = s_{n,j + m + \ell + \sum_{i=1}^{q-1}d_{i}}^{(k - \alpha_{q})} - s_{n,j + m + \sum_{i=1}^{q-1}d_{i}}^{(k - \alpha_{q})} - s_{n,j + \ell + \sum_{i=1}^{q-1}d_{i}}^{(k - \alpha_{q})} + s_{n,j + \sum_{i=1}^{q-1}d_{i}}^{(k - \alpha_{q})}.
\]
Now each $Y$ depends only on the $2(\ell + m)$ coordinates numbered $b_{j+\sum d_{i} + 1}, \ldots, b_{j+\sum d_{i} + \ell + m}$ and also $b_{j+\sum d_{i} + k - \alpha_{q} + 1}, \ldots, b_{j+\sum d_{i}+k-\alpha_{q}+\ell+m}$ and so we can apply the same argument as in Proposition \ref{P:weakpowerergodic} making use of Lemma \ref{L:key2} (since the $\alpha$ are small compared to $k$ and therefore do not significantly affect the number of $Y$ that depend on each of the coordinates).  Details are left to the reader since step-by-step the argument is the same as that for proving weak power ergodicity (the main point being that the $\alpha$ do not really change anything and without them the statement is identical to that in the weak power ergodicity proof).
\end{proof}

\begin{lemma}\label{L:weird}
For any sequence of complex numbers $c_{N,n}$ with $|c_{N,n}| \leq 1$ and any $0 \leq w_{N,n} \leq 1$ such that $\sum_{n=1}^{N} w_{N,n} = 1$
\[
\lim_{N \to \infty} \sum_{n=1}^{N} w_{N,n}|c_{N,n}| = 0 \quad\quad\text{if and only if}\quad\quad \lim_{N\to\infty}\sum_{n=1}^{N} w_{N,n}|c_{N,n}|^{2} = 0
\]
\end{lemma}
\begin{proof}
The right implies the left by the Cauchy-Schwarz Inequality and the left implies the right since $|c_{N,n}|^{2} \leq |c_{N,n}|$.
\end{proof}

In fact, the proof of Theorem \ref{T:rsmix} (with additional notation, but essentially line for line) also gives:
\begin{theorem}\label{T:rsmix2}
Let $\{ b_{n} \}$ be a sequence of (not necessarily iid) aperiodic random variables with finite mean (not necessarily uniformly bounded over $n$) and $\{ r_{n} \}$ be a sequence of positive integers with $r_{n} \to \infty$ where the $r_{n}$ depend only on $s_{m,j}$ for $m < n$ and $\mathbb{E}b_{n}$ is bounded by some polynomial in $r_{n}$ and such that almost every stochastic staircase transformation generated by $\{ b_{n} \}$ with cut sequence $\{ r_{n} \}$ is defined on a finite measure space.  Then almost every stochastic staircase transformation generated by $\{ b_{n} \}$ with cut sequence $\{ r_{n} \}$ is mixing.
\end{theorem}

\begin{remark}
The reader familiar with probability theory will note that in fact we only need that the $b_{n}$ be permutable (see, for instance, \cite{durrett}), i.e. invariant in distribution under permutations of a finite number of coordinates, for our result and do not need the full power of iid (permutability implies identical distribution and independence when conditioned on the right $\sigma$-algebra).
\end{remark}

\section{Examples of Mixing Transformations}

We conclude the paper by presenting a series of examples of transformations that our results imply are mixing, including the staircase transformations and Ornstein's constructions.

\subsection{Ornstein's Construction}\label{S:ornstein}

We conclude the paper by placing Ornstein's construction of mixing rank-one transformations in the context of our result.  Let $\{x_{n,j}\}$ for $j\in\{0,\ldots,r_{n}-1\}$ be iid uniform random variables on the set $\{-t_{n},\ldots,t_{n}\}$ where $t_{n}$ is small compared to $h_{n}$.  Set $s_{n,j} = 2h_{n-1} + x_{n,j+1} - x_{n,j}$.  The rank-one transformations with $\{s_{n,j}\}_{\{r_{n}\}}$ as spacer sequences with $t_{n} = h_{n-1}$ are Ornstein's original construction (see \cite{Or72}).  His result is that if $r_{n} \to \infty$ sufficiently fast then almost surely such transformations are mixing.
Set $b_{n,j} = s_{n,j} - s_{n,j-1} = x_{n,j+1} - 2x_{n,j} + x_{n,j-1}$.  Then the $b_{n,j}$ are a permutable sequence of aperiodic random variables with finite first moment bounded by $4t_{n}$ since the $x_{n,j}$ are iid uniform.  Thus a rank-one transformation with spacer sequence $\{s_{n,j}\}_{\{r_{n}\}}$ is a stochastic staircase transformation (per our remarks about stochastically generated dynamical sequences).  Our main theorem then implies that such transformations are mixing provided only that $t_{n}$ is bounded by some polynomial in $r_{n}$ and that $r_{n} \to \infty$ (a much more relaxed condition than in Ornstein's paper).
Variables with distribution given by $X - 2Y + Z$ where $X,Y,Z$ are uniform iid are not uncommon in probability theory and are precisely what gives the Ornstein construction. 
If we apply our result directly to $b_{n}$ being a uniform variable we obtain mixing transformations somewhere between random staircases and Ornstein's construction (sums of uniform variables as spacers).

\subsection{Random Polynomial Staircases}

In \cite{CS10} Section 7, it is also shown that \textbf{polynomial staircase transformations}, those with spacer sequence given by $s_{n,j} = p_{n}(j)$ where the $p_{n}$ are polynomials (of bounded degree), are mixing.  The proof makes use of the van der Corput Inequality to induct on the degree of the polynomials (the usual staircases being the base case).  Our work here also uses the van der Corput trick in a different way.
Without going into detail, we remark that it is possible to combine these two approaches and show that \textbf{random polynomial staircase transformations} are mixing.  By this we mean that choosing $s_{n,j} = p_{n}(b_{n,1}, \ldots, b_{n,j})$ to be some polynomial of bounded degree in the coordinates $b_{n,1},\ldots,b_{n,j}$ also leads to mixing.  The idea is to first perform the polynomial induction type step using van der Corput and then apply van der Corput twice more as we did above.  Ornstein's transformations can be viewed as a simple version of this very general construction.

\normalsize
\bibliography{StochasticStaircases}

}
\end{document}